\documentclass[12pt,reqno]{amsart}

\newcommand\version{February 17, 2022}

\usepackage[T1]{fontenc}
\usepackage[utf8]{inputenc}
\usepackage{amsmath,amsfonts,amsthm,amssymb,amsxtra,tikz,graphicx, float}
\usepackage{bbm} 
\usepackage[parfill]{parskip}    

\usepackage[english]{babel} 

\newtheorem{theorem}{Theorem}[section]
\newtheorem{proposition}[theorem]{Proposition}
\newtheorem{lemma}[theorem]{Lemma}

\newtheorem*{theoremA}{Theorem A}
\newtheorem*{theoremB}{Theorem B}

\theoremstyle{definition}

\newtheorem{conjecture}[theorem]{Conjecture}

\theoremstyle{remark}

\newtheorem{remark}[theorem]{Remark}

\numberwithin{equation}{section}

\newcommand{\C}{\mathbb{C}}

\renewcommand{\epsilon}{\varepsilon}

\newcommand{\N}{\mathbb{N}}

\newcommand{\R}{\mathbb{R}}

\newcommand*\diff{\mathop{}\!\mathrm{d}} 
\newcommand{\eps}{\varepsilon}

\newcommand{\Sph}{{\mathbb{S}^n}}
\newcommand{\sph}{{\mathbb{S}^{n-1}}}

\begin{document}

\title[Constant $Q$-curvature metrics with a singularity]{Constant $Q$-curvature metrics with a singularity}

\author{Tobias K\"onig}
\address[Tobias K\"onig]{Institut de Mathématiques de Jussieu, Université de Paris, Bâtiment Sophie Germain, Boite Courrier 7012,	8 Place Aurélie Nemours, 75205 Paris Cedex 13, France}
\email{koenig@imj-prg.fr}

\author{Paul Laurain}
\address[Paul Laurain]{
Institut de Mathématiques de Jussieu, Université de Paris, Bâtiment Sophie Germain, Case 7052, 75205
Paris Cedex 13, France \& DMA, Ecole normale supérieure, CNRS, PSL Research University, 75005 Paris.}
\email{paul.laurain@imj-prg.fr}

\thanks{\copyright\, 2021 by the authors. This paper may be reproduced, in its entirety, for non-commercial purposes.\\
\emph{Date}: \version \\
Partial support through ANR BLADE-JC ANR-18-CE40-002 is acknowledged. The authors thank Luca Martinazzi for several helpful comments on a first version of this article.}

\begin{abstract}
For dimensions $n \geq 3$, we classify singular solutions to the generalized Liouville equation $(-\Delta)^{n/2} u = e^{nu}$ on $\mathbb R^n \setminus \{0\}$ with the finite integral condition $\int_{\R^n} e^{nu} < \infty$ in terms of their behavior at $0$ and $\infty$. These solutions correspond to metrics of constant $Q$-curvature which are singular in the origin. Conversely, we give an optimal existence result for radial solutions.  This extends some recent results on solutions with singularities of logarithmic type to allow for singularities of arbitrary order. As a key tool to the existence result, we derive a new weighted Moser--Trudinger inequality for radial functions.
\end{abstract}

\maketitle

\section{Introduction}

Let $n \geq 2$. Our goal is to understand the structure of the set of solutions to the equation 
\begin{equation}
\label{liouville singular}(-\Delta)^{n/2} u = e^{nu} \qquad \text{ on } \R^n \setminus \{0\}, \qquad \Lambda := \int_{\R^n} e^{nu} \diff y < \infty,
\end{equation}
which may present a singularity at the origin.

For a function $u \in L^1_\text{loc}(\R^n)$, the expression $(-\Delta)^{n/2} u$ is to be understood as the tempered distribution satisfying 
\begin{equation}
\label{distributional definition}
\langle (-\Delta)^{n/2} u, \varphi \rangle = \int_{\R^n} u(x) (-\Delta)^{n/2} \varphi(x) \diff x 
\end{equation} 
for every $\varphi \in C^\infty_c(\R^n)$. If $n$ is odd, the hypothesis $\int_{\R^n} \frac{|u(x)|}{1 + |x|^{2n}} \diff x < \infty$ needs to be added, since $(-\Delta)^{n/2} \varphi$ is not compactly supported in these cases. We refer to \cite{HyMaMa} for more details and basic regularity results. In fact, every solution to \eqref{liouville singular} belongs to $C^\infty(\R^n \setminus \{0\})$, by the proof of \cite[Theorem 2.1]{HyMaMa}. 

For $n =2$, equation \eqref{liouville singular} is the classical Liouville equation, whose solutions $u$ correspond to metrics of constant Gauss curvature on $\R^2 \setminus \{0\}$. In higher dimensions, equation \eqref{liouville singular} plays a similar role in connection with the notion of $Q$-curvature. Indeed if one considers a smooth compact Riemannian surface $(\Sigma,g)$ with Gauss curvature $K_g$, then by a conformal change of metric $g_u=e^{2u} g$, the curvature changes as follows
$$-\Delta_g u + K_g= K_{g_u} e^{2u}.$$
Here, $\Delta_g= \frac{1}{\sqrt{\vert g\vert }}\partial_i (\sqrt{\vert g\vert }g^{ij} \partial_j)$ is the Beltrami Laplacian, which is a conformally invariant operator in dimension $2$ in the sense that $\Delta_{g_u}=e^{-2u}\Delta_g$. Relying on this conformal invariance property, Paneitz \cite{Pa} (and then Graham, Jenne, Mason and Sparling \cite{GJMS}) proved that, on a given Riemannian manifold $(M^n,g)$, with $n\geq 2$, there exists a unique differential\footnote{Pseudodifferential when the dimension is odd.} operator $P_n$ of order $n$ such that
$$ P^n_g = (-\Delta_g)^\frac{n}{2} +\text{lower order operator},$$
and
$$P^n_{g_u}=e^{-nu}P^n_g .$$
For $\xi$ the Euclidean metric on $\R^n$, we can check that $P_\xi^n=(-\Delta)^{n/2}$ on $(\R^n,\xi)$.
Hence, the notion of Gauss curvature is generalized by the one of $Q$-curvature, in the sense that we have the following identity
$$P^n_g(u)+ Q_g=e^{nu}Q_{g_u},$$
where $Q_g$, the $Q$-curvature, depends only on the curvature and its derivatives, for instance for $n=4$, we have
$$Q_g=-\frac{1}{6}\left( \Delta_g R_g -R_g^2 + 3 \vert \mathrm{Ric}_g\vert^2 \right),$$
where $R_g$ and $\mathrm{Ric}_g$ are the scalar and the Ricci curvature. On the sphere $\Sph$ the Paneitz operator has a simple expression in terms of the Laplacian and its eigenvalues, see Section \ref{section sphere}. For instance for $n=4$, we have, for $g_c$ the standard metric of $\mathbb S^4$,
$$P_{g_c}^4 =-\Delta_{g_c}(-\Delta_{g_c}+2),$$
see \cite{Chang} for more geometrical details.   

Hence, solutions to \eqref{liouville singular} correspond to metrics which are conformal to the Euclidean metric on $\R^n \setminus \{0\}$ and have constant $Q$-curvature equal to one. When the behaviour of $u$ near $0$ is $\displaystyle u\sim \gamma \ln(\vert x\vert)$, the metric $e^{nu} \xi$ can be interpreted as a conical metric, see \cite{FaMa}, but if the $u$ blow faster at the origin the nature of singularity is closer to that of an essential singularity. This is to our knowlege an inexplored field from the geometric point of view.  

A fundamental property of equation \eqref{liouville singular}, which we shall use several times in this paper, is its conformal invariance. In particular, if $u$ solves \eqref{liouville singular}, then its inversion
\begin{equation}
\label{u bar}
\bar{u}(x)= u\left(\frac{x}{|x|^2}\right)  -2 \ln (|x|).  
\end{equation}
is also a solution to \eqref{liouville singular}, of same mass $\Lambda$ as $u$.

\subsection{Overview of the problem}

In the past decades, entire solutions to the equation
\begin{equation}
\label{liouville fullspace}(-\Delta)^{n/2} u = e^{nu} \qquad \text{ on } \R^n, \qquad \Lambda := \int_{\R^n} e^{nu} \diff y < \infty,
\end{equation}
have been intensely studied for all values of $n \in \N^*$.

A fundamental observation is that the family of 'bubble' functions
\begin{equation}
\label{sph_sol}
u(x) = \ln \left( \frac{2 ((n-1)!)^{1/n} \lambda}{1 + \lambda^2 |x-x_0|^2} \right), \qquad \lambda > 0, \quad x_0 \in \R^n,
\end{equation}
solve \eqref{liouville fullspace}, for any $n \in \N^*$. These solutions all have the same mass 
\begin{equation}
\label{Lambda1 definition} 
\Lambda \equiv (n-1)! \int_{\R^n} \left(\frac{2}{1+|x|^2} \right)^n \diff x  = (n-1)! |\mathbb S^n|  =: \Lambda_1. 
\end{equation}
By the classification result of Chen and Lin \cite{ChLi}, these are all solutions to \eqref{liouville fullspace} in the classical case $n = 2$. If $n = 1$, the same is true, see \cite{LiMa} and references therein.

It has been put to our attention by Pierre-Damien Thizy that in fact the $n=2$ case was already essentially known by Liouville himself, the proof relies on the interpretation of this problem in terms of holomorphic functions, see \cite[Theorem 1]{CW} and \cite[p. 27]{Bandle}.

An important feature of equation \eqref{liouville fullspace}, which was first noticed in the fundamental paper \cite{Lin} by Lin, is that this classification results ceases to be true in higher dimensions $n \geq 3$. To explain this phenomenon in more detail, for a solution $u$ to \eqref{liouville fullspace}, we introduce its 'normalized version'
\begin{equation}
\label{definition_v}
v(x) := \frac{1}{\gamma_n} \int_{\R^n} \ln \left(\frac{|y|}{|x-y|} \right) e^{nu(y)} \diff y. 
\end{equation}
Here the constant 
\[ \gamma_n := \frac{(n-1)!}{2} |\mathbb S^n| = \frac{\Lambda_1}{2} \] 
is chosen such that $(-\Delta)^{n/2}  \ln \left( \frac{1}{|x|}\right) = \gamma_n \delta_0$ in the sense of distributions. 

\begin{theoremA}[\cite{Lin, Martinazzi, JiMaMaXi, Hyder-structure}]
\label{thm_lin_martinazzi}
Let $n \geq 3$ and let $u$ be a solution to \eqref{liouville fullspace}. Then there exists a polynomial $p$, bounded from above, of even degree at most $n-1$, such that
\[ u(x) = v(x) + p(x), \qquad x \in \R^n, \]
where $v$ is as in \eqref{definition_v}.

Moreover, as $|x| \to \infty$, 
\begin{equation}
\label{v asymptotics} 
v(x) = - \frac{\Lambda}{\gamma_n} \ln (|x|) + o(\ln |x|) 
\end{equation}

If $u(x) = o(|x|^2)$ as $|x| \to \infty$, or if $n=3,4$ and $\Lambda = \Lambda_1$, then $u$ is necessarily of the form \eqref{sph_sol}. 

Finally, if $n = 3,4$, then necessarily $\Lambda \leq \Lambda_1$. 
\end{theoremA}

Conversely, the following summary of existence results shows that the classification from Theorem \ref{thm_lin_martinazzi} is essentially optimal. 

In fact, it turns out that one can simultaneously prescribe the mass $\Lambda$ and its asmyptotic polynomial $p$ at infinity, provided $\Lambda  \in (0, \Lambda_1)$ and $p$ fulfills  
\begin{equation}
\label{condition P} 
p(x) \to -\infty \quad \text{ as } |x| \to \infty, \qquad \deg p \leq n-1. 
\end{equation}

\begin{theoremB}[\cite{ChCh, WeYe, HyMa, Hyder-existence, Martinazzi2013, Hyder2017}]
\label{thm_existence_polynomial}

Let $n \geq 3$. Given any $\Lambda  \in (0, \Lambda_1)$ and any polynomial $p$ satisfying \eqref{condition P}, there exists a solution $u$ to \eqref{liouville fullspace} such that
\[ u(x) = p(x) - \frac{\Lambda}{\gamma_n} \ln (|x|) + o(\ln |x|) \qquad \text{ as } \, |x| \to \infty. \]

Let $n \geq 5$. Given any $\Lambda \in (0, \infty)$ and any radially symmetric polynomial $p$ satisfying \eqref{condition P}, there exist numbers $c_1, c_2 > 0$ and a solution $u$ to \eqref{liouville fullspace} such that 
\[ u(x) = p(x) + c_1 |x|^2 - c_2 |x|^4 - \frac{\Lambda}{\gamma_n} \ln (|x|) + o(\ln |x|) \qquad \text{ as } \, |x| \to \infty. \]
\end{theoremB}

These findings are in remarkable contrast to the equation $(-\Delta)^k u = u^\frac{n+2k}{n-2k}, u > 0$ on $\R^n$, with $n/2 > k \in \N^*$, which is a closely related conformally invariant of \eqref{liouville singular}. Indeed, for this equation the classification results in \cite{CaGiSp, Lin, WeXu} guarantee that every solution has the standard bubble form corresponding to \eqref{sph_sol}, for every $n \geq 2$ and $k <n/2$.

Compared to the case of entire solutions to \eqref{liouville fullspace}, until recently singular solutions to \eqref{liouville singular} have received much less attention, with the exception of $n = 2$.

If $n = 2$, it has been shown in \cite{PrTa} that for \eqref{liouville singular} a classification result analogous to the entire case holds. Indeed, every solution is of the form 
\[ w(z) = \ln \left( \frac{2 (\alpha +1)\lambda}{1 + \lambda^2 |z^{\alpha+1} - \zeta|^2} \right), \qquad \lambda > 0, \,  \zeta \in \C, \, \alpha > -1, \]
in complex notation, and where necessarily $\zeta = 0$ if $\alpha \notin \N_0$. 

Still when $n = 2$, we mention the papers \cite{MaRu, CaMa}, where existence of solutions to a singular equation corresponding to \eqref{liouville singular} on general Riemannian surfaces has been derived using refined variational arguments and improved Moser-Trudinger inequalities.

\subsection{Main results}
\label{subsection main results}

The main point of our results is to give a version of both the classification and existence results from Theorems A and  B for the case of a general point singularity, i.e. for equation \eqref{liouville singular} instead of \eqref{liouville fullspace}.

Similar to before, for any solution $u$ to \eqref{liouville singular}, we define
\begin{equation}
\label{definition v}
v_u(x) := v(x) := \frac{1}{\gamma_{n}} \int_{\R^{n}} \ln \left( \frac{1 + |y|}{|x-y|} \right) e^{n u(y)} \diff y.
\end{equation}

The following theorem classifies all singular solutions to \eqref{liouville singular} in terms of their asymptotic behavior at 0 and infinity. 

\begin{theorem}
\label{theorem classification}
Let $n \geq 3$ and let $u$ be a solution to \eqref{liouville singular}. Then there exist $\beta \in \R$ and upper-bounded polynomials $p$ and $q$, of degree at most $n-1$, such that 
\begin{equation}
\label{eq classification thm} 
u(x) = v(x) + p(x) + q\left(\frac{x}{|x|^2}\right) + \beta \ln(|x|). 
\end{equation}
Moreover 
\[ \lim_{|x| \to \infty} \frac{v(x)}{\ln(|x|)} = - \frac{\Lambda}{\gamma_{n}} \qquad \text{ and } \qquad \lim_{x \to 0} \frac{v(x)}{\ln(|x|)} = 0. \]
\end{theorem}

For solutions to a special case of equation \eqref{liouville singular}, namely
\begin{equation}
\label{pde_sing_lit}
(-\Delta)^{n/2} u = e^{nu} -  \beta \gamma_n \delta_0,\qquad \int_{\R^n} e^{nu} < \infty,
\end{equation}
Theorem \ref{theorem classification} has been proved in \cite{HyMaMa}. However, even without appealing to the sophisticated existence result in Theorem \ref{theorem existence optimal} below, it is easy to see that there are solutions to \eqref{liouville singular} which do not satisfy \eqref{pde_sing_lit}. For instance, for $u$ is as in Theorem B with $p(x) = -|x|^2$, the inversion $\bar{u}$ defined in \eqref{u bar} solves \eqref{liouville singular} and has $q(x) = -|x|^2$ in \eqref{eq classification thm}. On the other hand, by \cite{HyMaMa} every solution to \eqref{pde_sing_lit} behaves like $\beta \ln |x|$ near zero and thus has necessarily $q \equiv 0$. For $n = 4$, by ODE arguments in the spirit of \cite{FaMa}, one can also easily construct solutions to \eqref{liouville singular} with both $p$ and $q$ non-zero, which thus cannot fulfil \eqref{pde_sing_lit} even after inversion.

Let us now discuss existence results for \eqref{liouville singular} more systematically. It is apparent from the results stated above that a major challenge is to prescribe the mass $\Lambda > 0$ of the solution, simultaneously with its asymptotic behavior, for $\Lambda$ in the largest possible range.

Our next theorem shows that in the singular setting of equation \eqref{liouville singular}, this can be done for every possible value of $\Lambda$, at least for radial functions.

\begin{theorem}
\label{theorem existence optimal}
Let $n \geq 3$, and let $p$ be a radial polynomial satisfying \eqref{condition P}. Suppose either that
\begin{enumerate}
\item[(a)] $q$ is another radial polynomial satisfying \eqref{condition P}, $\beta \in \R$ and $\Lambda > 0$ or that 
\item [(b)] $q \equiv 0$, $\beta > -1$ and $\Lambda \in (0, \Lambda_1(1 + \beta))$. 
\end{enumerate}

Then there exists a solution $u$ to \eqref{liouville singular} satisfying 
\[ u(x) = q\left(\frac{x}{|x|^2}\right) + \beta \ln (|x|) + o(\ln (|x|)) \qquad \text{ as } |x| \to 0 \]
and
\[ u(x) = p(x) + \left(\beta - \frac{\Lambda}{\gamma_n}\right) \ln (|x|) + o(\ln (|x|)) \qquad \text{ as } |x| \to \infty. \]
\end{theorem}

As in the case of entire solutions, this shows that the space of singular solutions to \eqref{liouville singular} becomes more complex as the order of derivative increases. Again, this is in stark contrast to $(-\Delta)^k u = u^\frac{n+2k}{n-2k}, u > 0$ on $\R^n \setminus \{0\}$ for $n/2 > k \in \N^*$. Indeed, when $k = 2$ it is proved in \cite{FrKo} that all solutions to this equation are given by a two-parameter family of functions $u_{a,T}(x) = |x|^{(4-n)/2} v_a(\log|x| + T)$, which presents no additional degrees of freedom compared to the basic case $k=1$ analyzed by \cite{CaGiSp, Schoen}. 


Part (b) of Theorem \ref{theorem existence optimal}, with $p(x) = -|x|^2$, is proved in \cite[Theorem 1.2]{HyMaMa}. Our more general version answers, for the case of radial $p$, an open question mentioned in \cite[Section 1.1]{HyMaMa}. Since in this case the functions one obtains grow logarithmically at the origin, these are indeed also solutions to \eqref{pde_sing_lit}. However, we use a different method of proof. Namely, the proof in \cite{HyMaMa} proceeds via a fixed point argument, while ours is based on a variational method, closer to the works \cite{ChCh} or \cite{Hyder-existence}. Our method allows to treat the case $q \not \equiv 0$ (which we did not succeed in via the fixed point method of \cite{HyMaMa}) and thus remove the upper bound on $\Lambda$. On the other hand, variants of this fixed point argument also yield existence results for non-radial solutions to equation \eqref{pde_sing_lit} \cite[Theorem 1.3]{HyMaMa}, respectively to its generalization $(-\Delta)^{n/2} u = e^{nu} - \sum_{l=1}^m\beta_l \gamma_n \delta_{P_l}$ with several singularities in points $P_1,...,P_m$ \cite[Theorem 1.2]{HyLiWe}, which have no counterpart in our paper. A variational existence argument in the spirit of our paper covering the case of several singularities when $n = 4$ can however be found in \cite{FaMa}. We also mention \cite{Maalaoui}, where existence results are derived by yet another different method.

The main tool in our approach to prove Theorem \ref{theorem existence optimal} is a weighted Moser-Trudinger inequality for radial functions on the sphere $\Sph$, which may be of independent interest and which we introduce now.   

Let $Q : \Sph \to [0, \infty)$ be a nonnegative weight function on $\Sph$ with the property that 
\begin{equation}
\label{Q decay}
 Q(\eta) \leq C d(\eta, N)^\beta \exp(-d(\eta, S)^{-\sigma}) 
\end{equation}
for some $\sigma, C > 0$, $\beta \in \R$. Here, we denote by $d(\eta, \xi)$ the geodesic distance between two points $\eta, \xi \in \Sph$ on the sphere. We use $N = e_{n+1}$ and $ S= -e_{n+1}$ to denote the north and south pole of $\Sph$.  

Moreover, for $s \in (0, n)$ we define 
\begin{equation}
\label{K definition} 
K_{n,s} =  \frac{\Gamma(\frac{n-s}{2})}{\Gamma(\frac{s}{2}) 2^s \pi^{n/2}}. 
\end{equation}

This constant is chosen such that the Green's function of $(-\Delta)^{s/2}$ on $\R^n$ is given by $K_{n,s} |x-y|^{-n+s}$. 

We prove the following inequality. 

\begin{theorem}
\label{theorem mt differential intro} 
Let $Q$ be as in \eqref{Q decay}. Then there is $C > 0$ such that 
\begin{equation}
\label{ineq corollary}
\int_\Sph \exp \left( \frac{n + \beta}{|\sph|} K_{n,\frac{n}{2}}^{-2}  u^2 \right) Q(\eta)  \leq C. 
\end{equation}
uniformly for $u \in C^\infty(\Sph)$ radial with $\|P_{n}^{1/2} u\|_{L^2(\Sph)} \leq 1$ and $\int_\Sph u = 0$. Moreover, the constant $ \frac{n+\beta}{|\sph|} K_{n,\frac{n}{2}}^{-2}$ is sharp in the sense that if it is replaced by $\gamma >  \frac{n+\beta}{|\sph|}K_{n,\frac{n}{2}}^{-2}$, then the constant $C$ is no longer uniform in $u$.
\end{theorem}

Theorem \ref{theorem mt differential intro} is in fact a special case of a more general inequality valid for any order of derivative strictly between $0$ and $n$. Indeed, the operator $P_n^{1/2}$ can be replaced by either $P_{s}$ with $s \in (0,n)$ or by $P_{2s}^{1/2}$ with $s \in (0,n/2)$, where $P_s$ is the Paneitz operator on $\Sph$ of order $s$; see \eqref{paneitz} for its definition. To keep the introduction focused, we defer a full statement to Theorems \ref{corollary mt differential bis} and \ref{corollary mt differential ter} below.

We stress that if $\beta > 0$, the best constant in the above theorems is strictly larger than the constant $\frac{n}{|\sph|} K_n^{-\frac{n}{n-s}}$ of the standard Moser-Trudinger inequality. In this case, the restriction to radial functions is thus truly essential. Indeed, the improved inequality must fail for a sequence of functions saturating the sharp constant in the unweighted inequality and concentrating in a point near which the weight $Q$ is regular. Since such a point is different from the poles, such functions must of course necessarily be non-radial.

A simple variation of the arguments leading to Theorem \ref{theorem mt differential intro} (respectively, Theorems \ref{corollary mt differential bis} and \ref{corollary mt differential ter}) yields a proof of a corresponding weighted Moser-Trudinger inequality for radial functions on balls of $\R^n$. 

\begin{theorem}
\label{thm mt local}
Let $s \in (0,n)$, $\beta \in \R$ and $R > 0$. Then there is $C = C(R) > 0$ such that 
\begin{equation}
\label{ineq mt local}
\int_{B_R} \exp \left( \frac{n+\beta}{|\sph|} K_{n,s}^{-\frac{n}{n-s}} |u|^\frac{n}{n-s} \right) |x|^\beta \diff x \leq C
\end{equation}
for every $u \in C^\infty_0(B_R)$ radial with $\|(-\Delta)^{s/2} u\|_{L^\frac{n}{s}(B_R)} \leq 1$. 
\end{theorem}

This generalizes the fractional Moser-Trudinger inequality of Martinazzi \cite{Martinazzi2015} to the weighted (and radial) setting. For even dimensions $n$ and $s = n/2$, Theorem \ref{thm mt local}, including the case $\beta > 0$, has been proved in the recent paper \cite{OMaOl} by a somewhat different method going back to \cite{Tarsi}. However, since this method does not pass through a fractional integral formulation (compare the discussion below and Theorem \ref{theorem mt convolution}), it does not seem to extend easily to the general case of fractional derivatives of arbitrary order. Similar weighted inequalities on all of Euclidean space appear in \cite{DoEsTa, dePiDo}.

Let us discuss in some more detail our proof strategy on the example of Theorem \ref{thm mt local}. The general strategy, going back to Adams \cite{Adams}, is to estimate $u$ pointwise by the convolution $v = G_s \ast (-\Delta)^{s/2} u$, where $G_s$ is the Green's function of $(-\Delta)^{s/2}$ and prove a corresponding estimate on the exponential integral of that convolution; see Theorem \ref{theorem mt convolution} below. 

When $\beta < 0$ in Theorem \ref{thm mt local}, the hypothesis of radial symmetry can be dropped by combining Martinazzi's Green's function estimate from \cite{Martinazzi2015} with the weighted fractional integral inequality of Lam and Lu \cite[Theorem 1.1]{LaLu}. Indeed, in this case the weight $|x|^\beta$ is radial-decreasing. The authors of \cite{LaLu}  can thus continue to follow Adams' proof, which consists in replacing $v$ by its symmetric decreasing rearrangement $v^*$ and using a variant of O'Neil's lemma \cite{ONeil} to get appropriate pointwise estimates on $v^*$. 

If $\beta > 0$, the weight $|x|^\beta$ is radial-\emph{increasing} and thus the rearrangement argument breaks down. Somewhat counterintuitively, this remains a challenge even if one restricts a priori to radial functions. Indeed, the terms which arise from transforming to radial variables turn out to have a slightly different structure than the O'Neil's-type estimate mentioned above, which leads to error terms of a different nature; see Remark \ref{remark oneils ineq} below for more details. Overcoming the non-availability of the rearrangement argument just described in the case $\beta > 0$ is in fact one of the main achievements of our proof. We are able to deal with the new error terms through a new improvement of the well-known Adams-Garsia's lemma \cite[Lemma 1]{Adams} stated as Lemma \ref{lemma adams} below.

In the setting on $\Sph$ of Theorem \ref{theorem mt differential intro} (respectively Theorems \ref{corollary mt differential bis} and \ref{corollary mt differential ter}), we strongly believe that when $-n <\beta < 0$, a symmetrization argument can equally allow to drop the radial symmetry assumption. To our knowledge, this has only been carried out rigorously for the special case $n = 4$ and $s = 2$ in the recent paper \cite{FaMa}.

\section{Classification results}
\label{section classification}

In this section we prove Theorem \ref{theorem classification}. We split the proof into two propositions to be proved in the following two subsections.

\subsection{The behavior of $u-v$}

The following proposition is the main result of this subsection. 

\begin{proposition}
\label{prop_u_decomposition}
Let $u$ solve \eqref{liouville singular} and let $v$ be defined by \eqref{definition v}. Then 
\[ u(x) - v(x) = p(x) + q\left(\frac{x}{|x|^2}\right) + \beta \ln (|x|) \]
for some polynomials $p$, $q$ and some $\beta \in \R$. Moreover, $p$ and $q$ are bounded from above with $\deg p, \deg q \leq n-1$. If $n$ is even, then $\deg p, \deg q \leq n-2$. 
\end{proposition}

\begin{lemma}
\label{lemma v equation}
The function $v$ defined in \eqref{definition v} satisfies 
\[ (-\Delta)^{n/2} v = e^{nu} \qquad \text{ on } \R^{n}  \]
in the sense of distributions.
\end{lemma}

\begin{proof}
This proof is very standard and amounts just to reproving that $\frac{1}{\gamma_{n}} \ln(\frac{1}{|x-y|})$ is the Green's function for $(-\Delta)^{n/2}$ on $\R^{n}$. To emphasize its generality, we write $f := e^{nu} \in L^1(\R^{n})$. Let $\varphi \in C^\infty_0(\R^{n})$. Then by Fubini
\begin{align*}
 \int_{\R^{n}} v(x) (-\Delta)^{n/2} \varphi(x) \diff x = \frac{1}{\gamma_{n}} \int_{\R^{n}} f(y) \left( \int_{\R^n} \ln \left( \frac{1+|y|}{|x-y|} \right) (-\Delta)^{n/2} \varphi(x) \diff x \right) \diff y . 
\end{align*} 
Firstly, we observe that
\[ \int_{\R^n} \ln ( 1 + |y|) (-\Delta)^{n/2} \varphi(x) \diff x = 0 \]
Then we are left with showing that
\[ \frac{1}{\gamma_{n}} \int_{\R^n} \ln \left( \frac{1}{|x-y|} \right) (-\Delta)^{n/2} \varphi(x) \diff x = \varphi(y). \]
Since $\varphi$ is smooth, this can be done by spliting the integral into the part over $B_\eps(y)$ and over its complement, integrate by parts and use that $(-\Delta)^i \frac{1}{|x-y|}$ is integrable on $\R^n$ for $i = 0,1,...,m-1$. 
\end{proof}

\begin{lemma}
\label{lemma v lower bound}
Let $u$ solve \eqref{liouville singular} and let $v$ be defined as in \eqref{definition v}. Then $v(x) \geq 0$ if $|x| \leq 1$. Moreover, for all $|x| \geq 1$ we have 
\begin{equation}
\label{eq_v_lowerbound} 
v(x) \geq - \frac{\Lambda}{\gamma_{n}} \ln(|x|).
\end{equation}
\end{lemma}

\begin{proof}
The proof of \eqref{eq_v_lowerbound} is identical to \cite[Lemma 3.1]{HyMaMa}: Let $|x| \geq 1$, then 
\[ |x-y| \leq |x| + |y| \leq |x|(1 + |y|) \]
and therefore
\[ \ln \left( \frac{1+|y|}{|x-y|} \right) \geq \ln \left(\frac{1}{|x|}\right) . \]
Similarly, if $|x| \leq 1$, we have 
\[ |x-y| \leq |x| + |y| \leq 1 + |y| \]
and therefore 
\[ \ln \left(\frac{1+|y|}{|x-y|}\right) \geq \ln (1) = 0. \]
Inserting these estimates into \eqref{definition v}, we obtain the conclusion. 
\end{proof}

We can now give the 

\begin{proof}
[Proof of Proposition \ref{prop_u_decomposition}]
We first prove the proposition under the assumption that $n \geq 4$ is even and then give the necessary modifications in case $n \geq 3$ is odd. 

\textit{Step 1. Bôcher's Theorem.    }
By Lemma \ref{lemma v equation}, we know that $\eta := u-v$ satisfies $(-\Delta)^{n/2} \eta = 0$ on $\R^{n} \setminus \{0\}$. In particular, $\eta$ is smooth on $\R^n \setminus \{0\}$. 

Then the generalized Bôcher theorem in \cite{FuKiMi} (applied on a sequence of balls $B_R(0)$ with radii $R$ tending to $\infty$) implies that 
\[ \eta = p +  \sum_{\alpha \leq s + n- 1} c_\alpha D^\alpha \ln \left(\frac{1}{|x|}\right) \qquad \text{ on } \quad \R^n \setminus \{0\}, \]
where $p \in C^\infty(\R^n)$ satisfies $(-\Delta)^{n/2} p = 0$ on $\R^n$ and $s$ is some non-negative integer such that 
\[ \int_{B_1(0) \setminus \{0\}} \eta^+(x) |x|^s \diff x < \infty. \]
Using Lemma \ref{lemma v lower bound} we can estimate, for $|x| \leq 1$, $\eta^+(x) \leq u^+(x) \leq e^{n u^+(x)}$. This is integrable, so we can take $s = 0$. Since it is easily proved by induction that $D^k \ln \left(\frac{1}{|x|}\right) = q_k\left(\frac{x}{|x|^2}\right)$ for every $|k|\geq 1$, where $q_k$ is a polynomial of degree at most $|k|$, we conclude that 
\begin{equation}
\label{expansion_u_proof}
 u(x) = v(x) + p(x) + q\left(\frac{x}{|x|^2}\right) + \beta \ln(|x|)
  \end{equation}
with $(-\Delta)^{n/2} p = 0$ on $\R^n$ and $\deg q \leq n -1$.

\textit{Step 2. Inversion.     }

Again, for definiteness suppose first that $n$ is even. 
We shall show in this step that $p$ must be a polynomial of degree at most $n -1$ as well. To do so, we make use of the conformal invariance of \eqref{liouville singular}. More precisely, the Kelvin transform $\bar{u}(x) := u\left(\frac{x}{|x|^2}\right) - 2 \ln (|x|)$ is also a solution to \eqref{liouville singular} with same mass $\Lambda$ and, by the above, can be written as 
\[ \bar{u}(x) = v_{\bar{u}}(x) + \tilde{p}(x) + \tilde{q}\left(\frac{x}{|x|^2}\right) + \tilde{\beta} \ln (|x|), \]
for some $\tilde{\beta} \in \R$, with $v_{\bar{u}}$ given by \eqref{definition v}. Moreover, a change of variables in the definition of $v$ shows the relation 
\begin{equation}
\label{v of u bar} 
v_{\bar{u}}(x) = v_u\left(\frac{x}{|x|^2}\right) - \frac{\Lambda}{\gamma_{n}}\ln(|x|). 
\end{equation}
Putting these identities together, we find 
\begin{equation}
\label{inversion_identity}
\tilde{p}(x) + \tilde{q}\left(\frac{x}{|x|^2}\right) + \left(\tilde{\beta} - \frac{\Lambda}{\gamma_{n}}\right)\ln (|x|) = p\left(\frac{x}{|x|^2}\right) + q(x) -\beta \ln (|x|) 
\end{equation} 
for all $x \neq 0$, where $(-\Delta)^{n/2} \tilde{p} = 0$ and $\deg \tilde{q} \leq n-1$. We claim that \eqref{inversion_identity} implies $\tilde{q} = p$. Indeed, by \eqref{inversion_identity} we have 
\[ \limsup_{|x| \to \infty} \frac{p(x)}{|x|^{n-1}} = \limsup_{|x|\to 0} |x|^{n-1} p\left(\frac{x}{|x|^2}\right) =  \limsup_{|x|\to 0} |x|^{n-1} \tilde{q}\left(\frac{x}{|x|^2}\right) < \infty \]
because $\deg \tilde{q} \leq n-1$. Thus $p(x) \leq 1 + |x|^{n-1}$ on $\R^n$ and the generalized Liouville theorem \cite[Theorem 5]{Martinazzi} implies that $p$ is a polynomial of degree $\deg p \leq n -1$. Once we know that both $\tilde{q}$ and $p$ must be polynomials, it is straightforward to deduce from \eqref{inversion_identity} that in fact $\tilde{q} = p$. 

(By exchanging the roles of $u$ and $\bar{u}$, we may of course deduce as well that $\tilde{p} = q$.)

\textit{Step 3.   Upper-boundedness and optimal degree.  }
We finish the proof by showing that $p$ and $q$ are bounded from above. In particular, $p$ and $q$ must be of even degree and thus $\deg p, \deg q \leq n -2$. Since $q = \tilde{p}$ as in Step 2, it suffices to prove that $p$ is bounded from above.

If $p$ is unbounded from above, inspired by Lemma 11 in \cite{Martinazzi}, then (see Theorem 3.1 in \cite{Gorin} and \cite[proof of Lemma 3.3]{HyMaMa}) there is $s > 0$ and a sequence $x_k$ with $|x_k| \to \infty$ such that $p(x_k) \geq |x_k|^s$. Since $\deg p \leq n-1$, we have $|\nabla p(y)| \lesssim |y|^{n-2}$ for all $|y|$ large enough. Thus, using also Lemma \ref{lemma v lower bound}, we have 
\[ u(x) \geq p(x) + \left(\beta- \frac{\Lambda}{\gamma_n} \right)\ln (|x|) - C \gtrsim |x_k|^s \]
for all $x$ in a ball $B_{\rho_k}(x_k)$ of radius $\rho_k \gtrsim |x_k|^\frac{s}{n-2}$. Therefore, for some $c, d > 0$
\[ \int_{\R^n} e^{n u}  dx\geq \int_{B_{\rho_k}(x_k)} e^{n u} dx\geq \rho_k^{n} e^{n c |x_k|^s} \geq d |x_k|^\frac{ns}{n-2}  e^{n c |x_k|^s} \to \infty \qquad \text{ as  } k \to \infty. \]
This contradiction proves that $p$ must be bounded from above. 

\textit{Step 4. $n$ odd.  } If $n$ is odd, we use that $(-\Delta)^\frac{n+1}{2} \eta = 0$ in Step 1 instead. Using the Bôcher theorem as above, we find that $u$ satisfies \eqref{expansion_u_proof}, but with $\deg q \leq n$. The argument from Step 2 then gives $\deg p \leq n$. Repeating Step 3, $p$ and $q$ are bounded from above. In particular, they are of even degree, hence $\deg p, \deg q \leq n-1$. 
\end{proof}

\subsection{The asymptotic behavior of $v$ at $\infty$}

The purpose of this subsection is to establish the leading-order behavior of $v(x)$ as $|x| \to \infty$. 

\begin{proposition}
\label{prop v asymptotics}
 $\lim_{|x| \to \infty} \frac{v_u(x)}{\ln(|x|)} = -\frac{\Lambda}{\gamma_{n}}$. 
\end{proposition}

We recapitulate in what follows the strategy in \cite{HyMaMa}, which carries over to our case and yields a proof of Proposition \ref{prop v asymptotics}. In Section \ref{subsec alternative asymptotics} below we present an alternative approach, which in our opinion is simpler and more direct, but which does not work in the general case, except if $n= 3,4$. 

Recall that by Lemma \ref{lemma v lower bound}, $v(x) \geq - \frac{\Lambda}{\gamma_n} \ln (|x|)$ for all $|x| \geq 1$. The following lemma yields a first step towards the desired reverse inequality. 

\begin{lemma}
\label{lemma v upper bound}
For every $\epsilon > 0$ there is $R > 0$ such that for all $|x| \geq R$ 
\begin{equation}
\label{v_bound_lin} v(x) \leq \left( -\frac{\Lambda}{\gamma_{n}} + \eps \right) \ln(|x|) + \int_{B(x,1)} \ln \left(\frac{1}{|x-y|} \right) e^{nu(y)} \diff y. 
\end{equation}
\end{lemma}

\begin{proof}
See \cite[proof of (2.11)]{Lin}. There $n=4$, but the proof is the same for general dimension $n \geq 3$, compare \cite[Lemma 9]{Martinazzi}
\end{proof}

The following lemma corresponds to \cite[Lemma 3.5]{HyMaMa}. (For simplicity, we only consider a radius equal to one.)   

\begin{lemma}
\label{lemma_average_v}
For every $t \in [1, \infty)$ and $\eps \in (0, \frac{\Lambda}{\gamma_n})$, there is $C> 0$ such that for all $x \in \R^n$, 
\[ \int_{B(x,1)} e^{tv(y)} \diff y \leq \frac{C}{|x|^{(\frac{\Lambda}{\gamma_n} - \eps)t}}. \]
\end{lemma}

\begin{proof}
[Proof of Lemma \ref{lemma_average_v}]
The proof is identical to \cite[Proof of Lemma 3.5]{HyMaMa}, with the only difference that the term $|y|^{n\beta} e^{nu(y)}$ is replaced by $e^{nu(y)}$. The proof works exactly the same because $\int_{\R^n} |y|^{n\beta} e^{nu(y)} \diff y$ is assumed to be finite in the normalization of \cite{HyMaMa}, where in our normalization this assumption becomes $\Lambda = \int_{\R^n} e^{nu(y)} \diff y < \infty$. 
\end{proof}

Lemma \ref{lemma_average_v} is already enough to prove Proposition \ref{prop v asymptotics} in the case when the coefficient $\beta$ of $\ln(|x|)$ in Proposition \ref{prop_u_decomposition} satisfies $\beta < \frac{\Lambda}{\gamma_n}$. 

\begin{proof}
[Proof of Proposition \ref{prop v asymptotics} in the case $\beta < \frac{\Lambda}{\gamma_n}$]
In view of Lemmas \ref{lemma v lower bound} and \ref{lemma v upper bound}, it remains to control the integral on the right side of \eqref{v_bound_lin}. For this purpose, we claim that there is $r > 1$, $R>0$ such that for all $|x| \geq R$,
\begin{equation}
\label{u_in_Lr}
\|e^{nu}\|_{L^r(B(x,1))} \leq C .
\end{equation}
Indeed, by Hölder's inequality this implies 
\begin{equation}
\label{v estimate Hölder} 
\int_{B(x.1)} \ln \left(\frac{1}{|x-y|} \right) e^{nu(y)} \diff y \leq \left( \int_{B(x,1)} \left( \ln \left(\frac{1}{|x-y|}\right) \right)^\frac{r}{r-1} \diff y \right)^\frac{r-1}{r} \|e^{nu}\|_{L^r(B(x,1))} \leq C. 
\end{equation}
Hence 
\[ \limsup_{|x| \to \infty} \frac{v(x)}{\ln(|x|)} \leq -\frac{\Lambda}{\gamma_{n}} + \eps \]
and since $\eps > 0$ was arbitrary, the proof is complete. 

We now prove \eqref{u_in_Lr}. Since $p$ and $q$ from Proposition \ref{prop_u_decomposition} are upper-bounded, and by Lemma \ref{lemma_average_v} applied with $q = nr$, 
\[ \int_{B(x,1)} e^{nru} \lesssim |x|^{n \beta r} \int_{B(x,1)} e^{nrv}  \lesssim |x|^{rn (\beta  - \frac{\Lambda}{\gamma_n} + \eps)}. \]
Since $\beta < \frac{\Lambda}{\gamma_n}$ by assumption, we can pick $\eps >0$ so small that $\beta - \frac{\Lambda}{\gamma_n} + \eps < 0$. Thus 
\[ \int_{B(x,1)} e^{nru} \to 0 \]
as $|x| \to \infty$. In particular \eqref{u_in_Lr} holds. 
\end{proof}

It remains to prove Proposition \ref{prop v asymptotics} in the case $\beta \geq \frac{\Lambda}{\gamma_n}$. The additional ingredient for this is the following Hölder-norm estimate.

\begin{lemma}
\label{lemma Hölder log}
Let $u$ solve \eqref{liouville singular} and let $v$ be defined by \eqref{definition v}. Suppose that $\beta \geq \frac{\Lambda}{\gamma_n}$, where $\beta$ is as in Proposition \ref{prop_u_decomposition}. Then as $|x| \to \infty$, 
\[ [v]_{C^{0, (\ln(|x|+1))^{-1}}(B(x,1))} = o(\ln (|x| + 1)). \] 
\end{lemma} 

\begin{proof}
Again, the proof is identical to \cite[proof of Lemma 3.6]{HyMaMa}, up to replacing the term $|y|^{n\beta} e^{nu(y)}$  by $e^{nu(y)}$. One proceeds by first proving the Campanato-type estimate
\[ \sup_{\rho \in (0, 4]} \frac{1}{\rho^{n + \frac{1}{\ln (|x|)}}} \int_{B(x, \rho)} \left | v(y) - \frac{1}{|B(x, \rho)|} \int_{B(x, \rho)} v(z) \diff z  \right| \diff y \to 0, \qquad \text{ as } |x| \to \infty, \]
and then transforming this to the desired Hölder-type bound by a standard argument. 

Let us point out that thanks to the classification result from Proposition \ref{prop_u_decomposition} and upper-boundedness of the polynomial $q$ the bound 
\[ e^{nu(y)} \lesssim |x|^{n \beta} e^{nv(y)} e^{np(y)}, \]
which is used in \cite[proof of Lemma 3.6]{HyMaMa} still holds. Also, the bound from \cite[Lemma 3.3]{HyMaMa} 
\[ \int_{B(x, \rho)} e^{qp} \lesssim |x|^{n(\frac{\Lambda}{\gamma_n} - \beta)}, \]
for every $q \geq 1$, $|x| \geq 1$ and $\rho$ small enough (uniformly in $x$ and $\rho$), still holds, by the same proof as given there. 
\end{proof}

We can now use Lemma \ref{lemma Hölder log} to give the proof of Proposition \ref{prop v asymptotics} in the remaining case $\beta \geq \frac{\Lambda}{\gamma_n}$. 

\begin{proof}
[Proof of Proposition \ref{prop v asymptotics} if $\beta \geq \frac{\Lambda}{\gamma_n}$]
By contradiction, in view of Lemma \ref{lemma v lower bound}, assume that there is $\delta > 0$ and a sequence $|x_k| \to \infty$ such that 
\[ v(x_k) \geq \left(-\frac{\Lambda}{\gamma_n} + 3 \delta \right) \ln (|x_k|). \]
By Lemma \ref{lemma Hölder log}, for any $y \in B(x_k,1)$, we have 
\[ v(y) = v(x_k) + o(\ln (|x_k|)) \geq \left(-\frac{\Lambda}{\gamma_n} + 2 \delta\right) \ln (|x_k|). \]
Thus 
\[ \int_{B(x_k,1)} e^{nv(y)} \diff y \geq |B_1| |x_k|^{-n(\frac{\Lambda}{\gamma_n} - 2 \delta)}. \]
On the other hand, by Lemma \ref{lemma_average_v} with $\eps = \delta$, 
\[ \int_{B(x_k,1)} e^{nv(y)} \diff y  \lesssim |x_k|^{-n(\frac{\Lambda}{\gamma_n} - \delta)}, \]
a contradiction. 
\end{proof}

\begin{proof}
[Proof of Theorem \ref{theorem classification}]
Theorem \ref{theorem classification} follows from Propositions  \ref{prop_u_decomposition} and \ref{prop v asymptotics}. 
\end{proof}

\subsection{The asymptotic behavior of $v$ at $\infty$: an alternative approach}
\label{subsec alternative asymptotics}

Given the relatively lengthy and involved proof of Proposition \ref{prop v asymptotics} in the previous subsection, we think that it is interesting to mention here an alternative strategy to find the asymptotic behavior of $v$ at $\infty$.  It exploits the conformal invariance of equation \eqref{liouville singular} and relies on the following simple-looking property of polynomials on $\R^n$.

\begin{conjecture}
\label{conjecture polynomials}
Let $q$ be a polynomial on $\R^n$. If $\sigma > -n$ is such that $\int_{\R^n \setminus B_1} |x|^\sigma e^{q(x)} \diff x < \infty$, then there is $\eps >0$ such that $\int_{\R^n \setminus B_1} |x|^{\sigma + \eps} e^{q(x)} \diff x < \infty$. 
\end{conjecture}


Admitting that Conjecture \ref{conjecture polynomials} is true, a straightforward proof of Proposition \ref{prop v asymptotics} goes as follows. 

\begin{proof}[Proof of Proposition \ref{prop v asymptotics} given Conjecture \ref{conjecture polynomials}]
By conformal invariance, $\bar{u} = u(x/|x|^2) - 2 \ln (|x|)$, is again a solution to \eqref{liouville singular}. In view of \eqref{v of u bar}, it suffices to prove that $v_{\bar{u}}$ is bounded near 0. 
After exchanging $\bar{u}$ for $u$, this follows from the definition \eqref{definition v} of $v$, if we can show. 
\begin{equation}
\label{e^nsu}
\int_{B_1} e^{nsu(x)} \diff x < \infty 
\end{equation}
for some $s > 1$. 

By the classification result from Proposition \ref{prop_u_decomposition}, we have
\begin{equation}
\label{e^nqu decomposition} 
e^{nsu} = e^{nsv} |x|^{n\beta s} e^{ns(p(x) + q(x/|x|^2))}. \end{equation}
Since $p$ and $q$ are bounded from above, the last factor is uniformly bounded on $B_1$. Moreover, arguing as in \cite[Proof of Theorem 2.1]{HyMaMa}, for every $t > 1$ we have
\begin{equation}
\label{e^nqv < infty}
\int_{B_1} e^{ntv(x)} \diff x < \infty. 
\end{equation} 
If $\beta > -1$,we thus have $e^{nsu} \in L^1(B_1)$ by Hölder, for $s > 1$ small enough that $\beta s > -1$.  

We may thus assume in what follows that $\beta \leq -1$ and observe that $v \geq 0$ on $B_1$, because 
\[ \ln \left(\frac{1+|y|}{|x-y|}\right) \geq 0 \qquad \text{ for every } \quad x \in B_1, \, y \in \R^n. \]
By integrability of $e^{nu}$, boundedness of $p$ and Lemma \ref{lemma v lower bound} we thus have
\begin{align*}
\infty &> \int_{\R^n} e^{nu} \geq \int_{B_1} |x|^{n\beta} e^{nq\left(\frac{x}{|x|^2}\right)} \diff x  = \int_{\R^n \setminus B_1} |x|^{-n(\beta+2)} e^{nq(x)} \diff x . 
\end{align*} 
By Conjecture \ref{conjecture polynomials}, there is $\eps > 0$ such that
\[ \infty >   \int_{\R^n \setminus B_1} |x|^{-n(\beta - \eps +2)} e^{nq(x)} \diff x = \int_{B_1} |x|^{n(\beta-\eps)} e^{-nq(\frac{x}{|x|^2})} \diff x. \]
Using this together with \eqref{e^nqv < infty} and upper-boundedness of $p$ and $q$, we easily infer from \eqref{e^nqu decomposition} and Hölder that $e^{n s u} \in L^1(B_1)$ for every $1 < s < \frac{\beta-\epsilon}{\beta}$. This completes the proof of \eqref{e^nsu}, and thus of Proposition \ref{prop v asymptotics}.  
\end{proof}

We are unfortunately not able to prove Conjecture \ref{conjecture polynomials} for general polynomials $q$. But we have the following partial result. 

\begin{proposition}
\label{proposition conjecture}
Suppose either that $\deg q \leq 3$, or that 
\begin{equation}
\label{assumption polynomial}
\begin{cases} & \text{ there is $k \in [0,n]$ such that $q(x)$ does not depend on $z := (x_{k+1}, ..., x_n)$} \\
& \text{ and $q(y) \to - \infty$ as $y := (x_1, ..., x_k) \to \infty$.} \end{cases} 
\end{equation}
Then Conjecture \ref{conjecture polynomials} holds. 
\end{proposition}

The choices $k= 0$ and $k=n$ in \eqref{assumption polynomial} cover the cases of polynomials which are constant, respectively tend to $-\infty$.

On the other hand, an example of a polynomial $q$ which is not covered by the assumption of Proposition \ref{proposition conjecture} is given by $q(x) =-x_1^2 x_2^2 - x_2^2 x_3^2 - ... - x_{n-1}^2 x_n^2$. This polynomial is bounded above and depends non-trivially on all variables, but it does not tend to $-\infty$ as $|x| \to \infty$ because it vanishes on every coordinate axis. 

We point out that if $n = 3,4$, then the asymptotic polynomial $q$ of any solution $u$ to \eqref{liouville singular} satisfies $\deg q \leq 2$ by Proposition \ref{prop_u_decomposition}. Therefore the alternative approach presented in this subsection yields a complete proof of Proposition \ref{prop v asymptotics} in dimensions $n = 3,4$.

\begin{proof}
[Proof of Proposition \ref{proposition conjecture}]

We claim that under assumption \eqref{assumption polynomial}, the integral 
$\int_{\R^n \setminus B_1} |x|^\sigma e^{q(x)} \diff x$ is finite if and only if
\begin{equation} 
\label{alpha bound}
\sigma < -n+k. 
\end{equation}
Clearly, this implies Conjecture \eqref{conjecture polynomials}. 

Let us first prove the 'only if' part. Writing $B_1^k$ for the unit ball of $\R^k$, we observe that $\R^k \setminus B_1^k \times \R^{n-k} \setminus B_1^{n-k} \subset \R^n \setminus B_1$. Therefore finiteness of the integral implies 
\[ \infty > \int_{\R^n \setminus B_1} |x|^{\sigma} e^{q(x)} \diff x \geq \int_{\R^k \setminus B^k_1} e^{q(y)} \left(  \int_{\R^{n-k} \setminus B_{n-k}} (|y|^2 + |z|^2)^{\sigma/2} \diff z \right) \diff y \]
Thus for a.e. $y \in \R^k \setminus B_1^k$, we must have 
\[ \int_{\R^{n-k} \setminus B_{n-k}} (|y|^2 + |z|^2)^{\sigma/2} \diff z < \infty. \]
which yields \eqref{alpha bound}. 

Conversely, suppose that \eqref{alpha bound} holds. Then 
\begin{align*}
\int_{\R^n \setminus B_1} |x|^\sigma e^{q(x)} \diff x & = \int_{\R^k \setminus B_1^k} e^{q(y)} \left( \int_{\R^{n-k}} (|y|^2 + |z|^2)^{\sigma/2)} \diff z \right) \diff y \\
& \qquad + \int_{B_1^k}  e^{q(y)} \left( \int_{\R^{n-k} \setminus B_{\sqrt{1 - |y|^2}}^{n-k}} (|y|^2 + |z|^2)^{\sigma/2} \diff z \right) \diff y.
\end{align*}
The first summand equals
\begin{align*}
 \int_{\R^k \setminus B_1^k}  e^{q(y)} |y|^{(n-k) +\sigma} \left( \int_{\R^{n-k}} (1 + |z|^2)^{\sigma/2} \diff z \right) \diff y < \infty,
\end{align*} 
since by \eqref{assumption polynomial}, \cite[Theorem 3.1]{Gorin} yields $q(y) \lesssim -|y|^s + C$ for some $s > 0$ as $|y|\to \infty$ on $\R^k$ and since by \eqref{alpha bound},
\[ \int_{\R^{n-k}} (1 + |z|^2)^{\sigma/2} \diff z  < \infty. \]

To bound the second summand, we need to control the $\diff z$-integral as $|y| \to 0$. We compute, for $|y|$ small enough,
\begin{align*}
& \qquad \int_{\R^{n-k} \setminus B_{\sqrt{1 - |y|^2}}^{n-k}} (|y|^2 + |z|^2)^{\sigma/2} \diff z \\
& \leq |y|^{n-k+\sigma} \int_{\R^{n-k} \setminus B_{\frac{\sqrt{1 - |y|^2}}{|y|^2}}} |z|^{\sigma/2} \diff z \\
&\leq  |\mathbb S^{n-k-1}| |y|^{n - k + \sigma} \left( \frac{2}{|y|} \right)^{n-k + \sigma}  \lesssim 1
\end{align*}
as $|y| \to 0$. Hence the second summand is bounded as well, which completes the proof of the claimed equivalence. 

Finally, assume that $\deg q \leq 3$. Arguing as in Step 3 of the proof of Proposition \ref{prop_u_decomposition}, we find that if $\int_{\R^n \setminus B_1} |x|^\sigma e^{nq(x)} \diff x < \infty$, then $q$ must be upper-bounded, and hence have even degree.  If $\deg q = 0$, assumption \eqref{assumption polynomial} holds with $k = 0$. If $\deg q = 2$, then $q$ can be written, up to an orthogonal coordinate transformation, as 
\[ q(x) = -\sum_{i=1}^n a_i(x_i - x_0)^2 + c_0, \]
for some $a_i  \geq 0$. Thus assumption \eqref{assumption polynomial} holds as well, with $k$ being the number of non-zero coefficients $a_i$. 
\end{proof}

\section{The setting on the sphere $\mathbb S^n$}
\label{section sphere}

To facilitate the variational argument leading to Theorem \ref{theorem existence optimal}, we shall pass to an equivalent problem on the sphere $\mathbb S^n = \{ x \in \R^{n+1} \, : \, |x| = 1 \}$. In this brief section, we introduce the necessary preliminaries for this. 

Let us denote $\mathcal S :\R^n \to \Sph \subset \R^{n+1}$ the inverse stereographic projection, i.e.
$$
\mathcal S_j(x) = \frac{2x_j}{1+|x|^2} \,,\ j=1,\ldots, n\,,
\qquad
\mathcal S_{n+1}(x) = \frac{1-|x|^2}{1+|x|^2} \,.
$$
We denote its Jacobian in a point $x \in \R^n$ by $J_{\mathcal S}(x) := (2/(1 + |x|^2))^n$ and that of $\mathcal S^{-1}$ in a point $\eta \in \Sph$ by $J_{\mathcal S^{-1}}(\eta)$. Moreover, we call a function $v$ on $\Sph$ \emph{radial} if $v(\mathcal S(x))$ is a radial function on $\R^n$.

We now introduce the differential operators $P_{2s}$ mentioned in the introduction. For this, recall that every $u \in L^2(\mathbb S^n)$ can be uniquely written as 
\begin{equation}
\label{spherical harm expansion} 
u = \sum_{l=0}^\infty \sum_{m = 1}^{N_l} u_{lm} Y_{lm} \qquad \text{ for some } u_{lm} \in \R. 
\end{equation}
Here $Y_{lm}$ are a fixed $L^2(\mathbb S^n)$-orthonormal basis of spherical harmonics, i.e. eigenfunctions of the Laplace-Beltrami operator $(-\Delta)_{\mathbb S^n}$ satisfying $(-\Delta)_{\mathbb S^n} Y_{lm} = \lambda_l Y_{lm}$, with $\lambda_l = l (l + n -1)$ being the $l$-th eigenvalue of multiplicity $N_l \in \N$. 
For $s > 0$, we introduce the $2s$-th order Paneitz operator $P_{2s}$ by defining 
\begin{equation}
\label{paneitz}
 P_{2s} u := \sum_{l,m} \frac{\Gamma(l+\frac n2 + s)}{\Gamma(l+\frac n2 - s)} u_{lm} Y_{lm}
\end{equation}
for every $u \in L^2(\mathbb S^n)$ such that the right side converges in $L^2(\mathbb S^n)$. Here and below, the sum is the same as in \eqref{spherical harm expansion} and we interpret $\Gamma(-n)^{-1} = 0$ for $n \in \N_0$. 
A conventional alternative expression for $P_{2s}$ in terms of $(-\Delta)_{\mathbb S^n}$ is 
\begin{align}
	\label{eq:opas}
	A_{2s} = \frac{\Gamma(B+\tfrac12 + s)}{\Gamma(B+\tfrac12 - s)},
	\qquad\text{for}\qquad
	B = \sqrt{-\Delta_{\Sph} + \tfrac{(n-1)^2}{4}} \,.
\end{align}
It is clear from \eqref{paneitz} that for $s \in (0, n/2]$, the operator $P_{2s}$ is nonnegative. Thus we may define its square root by 
\begin{equation}
\label{Pn^12}
P_{2s}^{1/2} u := \sum_{l,m} \left(\frac{\Gamma(l+\frac n2 + s)}{\Gamma(l+\frac n2 - s)} \right)^{1/2} u_{lm} Y_{lm} 
\end{equation}
for every $u$ such that the right side converges in $L^2(\Sph)$.

A crucial feature of $P_{2s}$ is its conformal transformation property (see e.g. \cite[eq. (4.11)]{Gonzalez})
\begin{equation}
\label{conf trafo Pn} 
(-\Delta)^{s} (J_{\mathcal S}^\frac{n-2s}{2n} u \circ \mathcal S) = J_{\mathcal S}(x)^\frac{n+2s}{2n} \left((P_{2s} u) \circ \mathcal S \right) (x) \qquad \text{ for every } u \in C^\infty(\mathbb S^n).  
\end{equation}
In particular, for $s = n/2$, writing $w =  u \circ \mathcal S$, \eqref{conf trafo Pn} implies 
\begin{equation}
\label{conf trafo Pn int}
\int_{\mathbb S^n} |P_n^{1/2} u|^2 = \int_{\mathbb S^n} u P_n u = \int_{\R^n} w \left((P_n u) \circ \mathcal S\right) J_{\mathcal S} = \int_{\R^n} w (-\Delta)^{n/2} w. 
\end{equation}
We finish by defining the Sobolev space 
\begin{equation}
\label{Hn2} 
H^{n/2}(\mathbb S^n) := \left \{ u \in L^2(\mathbb S^n) \, : \, \|u\|_{H^{n/2}(\mathbb S^n)}^2 := \|u\|^2_{L^2(\mathbb S^n)} + \|P_n^{1/2} u \|^2_{L^2(\mathbb S^n)}  < \infty \right \} 
\end{equation}
which will play a role in the proof of Theorem \ref{theorem existence optimal}. By \eqref{Pn^12} and Stirling's formula, the norm $ \|u\|_{H^{n/2}(\mathbb S^n)}$ is equivalent to the standard $H^{n/2}(\Sph)$ norm  $(\sum_{l,m} l^n u_{lm}^2)^{1/2}$. We denote by $H^{n/2}_\text{rad}(\Sph)$ the subspace of $H^{n/2}(\mathbb S^n)$ consisting of radial functions.

\section{Proof of Theorem \ref{theorem existence optimal}}
\label{section existence}

In this section, we prove Theorem \ref{theorem existence optimal} using a variational argument, which, similarly to earlier works like \cite{ChCh} and \cite{Hyder-existence}, it is convenient to carry out on the $n$-dimensional sphere $\mathbb S^n$. 

We divide the proof into several steps.


\begin{proof}
[Proof of Theorem \ref{theorem existence optimal}]

\textit{Step 1. Reduction to a variational equation. }

Let $\Lambda > 0$, $\beta \in \R$ and $p, q$ (radial polynomials of degree at most $n-1$ tending to $-\infty$) be given as in the statement of Theorem \ref{theorem existence optimal}. 

Moreover, fix a function $u_0 \in C^\infty(\R^n)$ such that $u_0(x) = - \ln (|x|)$ for all $|x| \geq 1$ and denote $\varphi_0 := (-\Delta)^{-n/2} u_0$. 

Let 
\[ K(x) := |x|^{n\beta} \exp\left( n \left( p(x) + q\left(\frac{x}{|x|^2}\right) + \frac{\Lambda}{\gamma_n} u_0\right) \right). \]
For $w \in H^{n/2}(\R^n)$, let $c_w \in \R$ be the unique number such that 
\[ \int_{\R^n} K e^{n (w + c_w)} = \Lambda, \qquad \text{ i.e. } \quad e^{n c_w} = \frac{\Lambda}{\int_{\R^n} K e^{nw}}. \]

Equipped with these notations, we claim that to prove Theorem \ref{theorem existence optimal}, it suffices to find $w \in H^{n/2}(\R^n)$ with $w(x) = o\left(\ln (|x|)\right)$ as $|x| \to \infty$ such that 
\begin{equation}
\label{equation w} 
(-\Delta)^{n/2} w = K e^{n(w + c_w)} - \frac{\Lambda}{\gamma_n} \varphi_0. 
\end{equation}
Indeed, then 
\[ u(x):= w(x) + p(x) + q\left (\frac{x}{|x|^2}\right)+ \beta \ln (|x|) + \frac{\Lambda}{\gamma_n} u_0(x) + c_w \]
fulfills 
\[ (-\Delta)^{n/2} u(x) = (-\Delta)^{n/2} w(x) + \frac{\Lambda}{\gamma_n} \varphi_0 = K e^{n(w + c_w)} = e^{nu} \]
as well as 
\[ \int_{\R^n} e^{nu} = \int_{\R^n} K e^{n(w + c_w)} = \Lambda. \] 

\textit{Step 2. Transforming to the sphere. }

Set
\[ Q := J_{\mathcal S^{-1}} K \circ \mathcal S^{-1}, \qquad \psi_0 := J_{\mathcal S^{-1}} \varphi_0 \circ \mathcal S^{-1} \]
so that $\int_{\R^n} K e^{nw} = \int_{\mathbb S^n} Q e^{nu}$. By \eqref{conf trafo Pn}, equation \eqref{equation w} is then equivalent to 
\begin{equation}
\label{equation u sphere}
P_n u = -\frac{\Lambda}{\gamma_n} \psi_0 + \Lambda \frac{Q e^{nu}}{\int_{\mathbb S^n} Q e^{nu}}. 
\end{equation}

Solutions to \eqref{equation u sphere} are given by critical points of the functional 
\[ I[u]:= \frac 12  \| P_n^{1/2} u\|^2 + \frac{\Lambda}{\gamma_n} \int_{\mathbb S^n} \psi_0 u - \frac{\Lambda}{n} \ln \left(\int_{\mathbb S^n} Q e^{nu} \right). \]

We will find such a critical point by solving the variational problem
\begin{equation}
\label{var prob sphere} 
 \inf \left\{ I[u] \, : \, u \in H^{n/2}_\text{rad}(\mathbb S^n) \right\}.
\end{equation}
Notice that the infimum is taken over radial functions only.

We claim that any minimizer $u_0$ of \eqref{var prob sphere} is a critical point of $I[u]$, i.e. for any test function $\phi \in H^{n/2}(\mathbb S^n)$, it satisfies 
\begin{equation}
\label{el eq u0 sphere}
\int_{\R^n} P_n^{1/2} u_0 P_n^{1/2} \phi + \frac{\Lambda}{\gamma_n} \int_{\mathbb S^n} \psi_0 \phi -  \frac{\Lambda}{\int_\Sph Q e^{nu_0}} \int_\Sph Q e^{nu_0} \phi = 0.
\end{equation}
If $\phi$ is itself radial, this is clear because $u_0$ is a minimizer over radial functions. The validity of \eqref{el eq u0 sphere} for general $\phi \in H^{n/2}(\Sph)$ follows since $u_0$, $\psi_0$ and $q$ are radial and because $P_n$ preserves radial functions.

\textit{Step 3. Minimizing $I[u]$. }

As explained in the previous step, it remains to prove that a minimizer for the problem \eqref{var prob sphere} exists.  We shall give the proof using the direct method of the calculus of variations, making crucial use of the improved Moser-Trudinger inequality  in the form of Theorem \ref{theorem mt differential intro}. 

Indeed, we claim that there is a constant $C >0$ only depending on $n$, $q$ and $\Lambda$ such that 
\begin{equation}
\label{I lower bound} 
I[u] \geq \frac 14 \|P_n^{1/2} u\|_2^2 - C
\end{equation}
for all $u \in H^{n/2}_\text{rad}(\mathbb S^n)$. To prove \eqref{I lower bound}, first note that $P_n^{1/2} c = 0$ and 
\begin{equation}
\label{I inv constants}
I[u+c] = I[u], \qquad \text{ for any constant  } \quad c \in \R,
\end{equation}
as a consequence of the fact that $\int_{\mathbb{S}^n} \psi_0 = \int_{\R^n} \varphi_0 = \gamma_n$, see \cite[Lemma 2.3]{Hyder-existence}. Thus we may assume $\int_{\mathbb S^n} u = 0$. For such $u$, the Poincaré-type inequality 
\begin{equation}
\label{poincare}
\|u\|_2^2 \leq C \|P_n^{1/2} u\|_2^2, 
\end{equation}
holds as a consequence of the definition of $P_n^{1/2}$ in \eqref{Pn^12}. Thus we can bound the subcritical term by
\[ \frac{\Lambda}{\gamma_n} \left| \int_{\mathbb S^n} \psi_0 u \right| \leq C \|\psi_0\|_2 \|P_n^{1/2} u\|_2 \leq \frac14 \|P_n^{1/2} u\|_2^2 + C \|\psi_0\|_2^2. \]

Next, we bound the critical term $\ln\left( \int_{\mathbb S^n} Q e^{nu}\right) $ as follows. For any $\gamma >0$, 
$nu \leq \frac{n^2}{4 \gamma} \|P_n^{1/2} u\|_2^2 + \gamma \frac{u^2}{ \|P_n^{1/2} u \|^2_2}$. Hence 
\begin{align*}
\int_{\mathbb S^n} Q e^{nu} \leq \exp \left( \frac{n^2}{4 \gamma} \|P_n^{1/2} u\|_2^2 \right) \int_{\mathbb S^n} Q e^{\gamma \frac{u^2}{ \|P_n^{1/2} u \|_2^2}}. 
\end{align*}
Thus, by Theorem \ref{theorem mt differential intro} with $s = n/2$, since $Q(\eta) \leq C\exp{-d(\eta,S)^{-\sigma}} d(\eta,N)^{\beta}$ for any $\beta>0$, we have for any $\gamma > 0$ that
\[ \frac{\Lambda}{n} \ln \left(\int_{\mathbb S^n} Q e^{nu}\right) \leq C + \frac{n \Lambda}{4 \gamma} \|P_n^{1/2} u \|_2^2.  \]
Choosing $\gamma > 0$ so large that $\frac{n \Lambda}{4 \gamma} < \frac{1}{4}$ and combining the above estimates, 
\eqref{I lower bound} follows. 

We can now prove that \eqref{var prob sphere} admits a minimizer. Let $(u_k)_{k \in \N} \subset H^{n/2}_\text{rad}(\mathbb S^n)$ be a minimizing sequence for \eqref{var prob sphere}. By \eqref{I inv constants}, we may again assume that $\int_{\mathbb S^n} u_k = 0$ for all $k \in \N$. By \eqref{I lower bound}, $\|P_n^{1/2} u_k\|_2$ is uniformly bounded. By \eqref{poincare}, so is $\|u_k\|_2$, and hence $(u_k)$ is uniformly bounded in $H^{n/2}_\text{rad}(\mathbb S^n)$. Up to extracting a subsequence, we may therefore assume $u_n \rightharpoonup u_0$ for some $u_0 \in H^{n/2}_\text{rad}(\mathbb S^n)$. 

Then 
\[ \|u_0\|_{H^{n/2}(\mathbb S^n)} \leq \liminf_{k\to \infty} \|u_k\|_{H^{n/2}(\mathbb S^n)}. \]
Moreover, since $u \mapsto u$ and $u \mapsto e^u$ defines compact embeddings from $H^{n/2}(\mathbb S^n)$ into $L^2(\mathbb S^n)$ (for the latter property see \cite[Proposition 7]{JiMaMaXi}), 
\[ \|u_0\|_2 =\lim_{k\to \infty} \|u_k\|_2, \quad \int_{\mathbb S^n} \psi_0 u_0 = \lim_{k\to \infty} \int_{\mathbb S^n} \psi_0 u_k \]
and 
\[   \ln \left(\int_{\mathbb S^n} Q e^{nu_0} \right) = \lim_{k \to \infty} \ln \left(\int_{\mathbb S^n} Q e^{nu_k} \right). \]
In view of the definition of $\|u\|_{H^{n/2}}$ in \eqref{Hn2}, this implies 
\[ \|P_n^{1/2} u_0\|_2 \leq \liminf_{k\to \infty}\|P_n^{1/2} u_k\|_2 \]
and thus $u_0$ is a minimizer for $I_0$.

\textit{Step 4. Regularity } 

The argument that provides regularity is identical to \cite{Hyder-existence}. Somewhat more precisely, by using the regularity theory for $P_n$ that can easily be deduced from its representation on spherical harmonics \cite[Lemmas 2.5 and 2.6]{Hyder-existence} together with the fact that $\psi_0 \in C^{2n+1}(\mathbb S^n)$, one proves $u_0 \in C^{2n+1}(\mathbb S^n)$. Since $u_0$ is continuous at the south pole, the function $w_0 := u_0 \circ \mathcal S$ has a limit at infinity, in particular the condition $w(x) = o\left(\ln (|x|)\right)$ as $|x| \to \infty$ is satisfied. Moreover, $w_0$ fulfills 
\[ (-\Delta)^{n/2} w_0 = K e^{n(w_0 + c_{w_0})} - \frac{\Lambda}{\gamma_n} \varphi_0. \]
Since $\varphi_0 \in C^\infty(\R^n)$, bootstrapping gives $w_0 \in C^\infty(\R^n)$. Thus we have found a solution $w = w_0$ as described in Step 1 and the proof is complete. 
\end{proof}

\begin{proof}
[Proof of Theorem \ref{theorem existence optimal}(b)]
The proof of part (b) is identical to the proof of (a), except for the use of Theorem \ref{theorem mt differential intro}. Indeed, in this case we have $Q(\eta) \leq C \text{dist}(\eta, N)^{n\beta} e^{-\text{dist}(\eta, S)^{-\sigma}}$ for some $\sigma > 0$ and some given $\beta > -1$. Theorem \ref{theorem mt differential intro} thus yields 
\[ \frac{\Lambda}{n} \ln \left(\int_{\mathbb S^n} Q e^{nu}\right) \leq C + \frac{n \Lambda}{4 \gamma} \|P_n^{1/2} u \|^2,  \]
with
\[ \gamma = \frac{n + \beta}{|\mathbb S^{n-1}|} K_{n, \frac{n}{2}}^{-2}. \]
Since $\Lambda < (1 + \beta) \Lambda_1$, recalling the numerical values of $\Lambda_1$ and $K_{n,n/2}$ given in \eqref{Lambda1 definition} and \eqref{K definition}, a direct computation shows that 
\[ \delta :=  \frac{1}{2} - \frac{n \Lambda}{4 \gamma} > 0. \]
The rest of the proof proceeds as above. 
\end{proof}

\section{Proof of Theorems \ref{theorem mt differential intro} and \ref{thm mt local}}
\label{section moser trudinger}

We now turn to prove the Moser-Trudinger inequalities from Theorems \ref{theorem mt differential intro} and \ref{thm mt local}. As briefly explained in the introduction, following Adams' classical paper \cite{Adams} we shall derive them from a dual inequality on convolution-type operators which we introduce now in detail.

Let $s \in (0, n)$. For $f \in L^{\frac{n}{s}}(\Sph)$, we denote  
\begin{equation}
\label{T definition} 
Tf(\eta) := \int_\Sph k(\eta, \xi) f(\xi) \diff \sigma(\xi).  
\end{equation}
Here we assume that the kernel $k(\eta, \xi)$ satisfies
\begin{equation}
\label{expansion K}
k(\eta, \xi) = d(\eta, \xi)^{-n+s} \left(1 + \mathcal O\left(d(\eta, \xi)\right)^\alpha \right). 
\end{equation} 
for some fixed constant $\alpha > 0$. (We denote by $\mathcal O( d(\eta, \xi)^\alpha)$ a quantity $h(\eta, \xi)$ with the property that $|h(\eta, \xi)| d(\eta, \xi)^{-\alpha}$ is bounded uniformly in $\eta, \xi \in \Sph$ with $\eta \neq \xi$.) 
 
Moreover, we assume that $k(\eta, \xi)$ only depends on the geodesic distance $d(\eta, \xi)$. As a consequence, by change of variables, $Tf$ is radial (in the sense given in Section \ref{section sphere}) if $f$ is. 

We are going to prove the following weighted Moser-Trudinger-Adams inequality of convolution type, valid for radial functions. 

\begin{theorem}
\label{theorem mt convolution}
 Let $Q$ be as in \eqref{Q decay} and let $T$ be given by \eqref{T definition}, for some $k$ satisfying \eqref{expansion K}. Then there is $C > 0$ such that
\[ \int_\Sph \exp \left( \frac{n + \beta}{|\sph|}  |Tf|^\frac{n}{n-s} \right) Q(\eta)  \leq C \]
uniformly for $f \in L^{\frac{n}{s}}_\text{rad}(\Sph)$ with $\|f\|_{L^{\frac{n}{s}}(\Sph)} \leq 1$. Moreover, the constant $ \frac{n+\beta}{|\sph|}$ is sharp, in the sense that if it is replaced by $\gamma >  \frac{n+\beta}{|\sph|}$, then the constant $C$ is no longer uniform in $f$. 
\end{theorem}

As in Adams' paper \cite{Adams}, we may deduce from this the Moser-Trudinger-Adams inequality in its differential form, more precisely Theorem \ref{thm mt local} and the following two theorems. Recall that the constant $K_{n,s}$ is given in \eqref{K definition} and that the operators $P_s$ and $P_{2s}^{1/2}$ are defined in \eqref{paneitz} and \eqref{Pn^12}.

\begin{theorem}
\label{corollary mt differential bis}
Let $s \in (0, \frac{n}{2}]$ and let $Q$ be as in \eqref{Q decay}. Then there is $C > 0$ such that 
\begin{equation}
\label{ineq corollary bis}
\int_\Sph \exp \left( \frac{n + \beta}{|\sph|} K_{n,s}^{-\frac{n}{n-s}}  |u|^\frac{n}{n-s} \right) Q(\eta)  \leq C. 
\end{equation}
uniformly for $u \in C^\infty(\Sph)$ radial with $\|P_{2s}^{1/2} u\|_{L^{n/s}(\Sph)} \leq 1$. If $s = n/2$, we assume additionally $\int_\Sph u = 0$. Moreover, the constant $ \frac{n+\beta}{|\sph|} K_n^{-\frac{n}{n-s}}$ is sharp in the sense that if it is replaced by $\gamma >  \frac{n+\beta}{|\sph|}K_{n,s}^{-\frac{n}{n-s}}$, then the constant $C$ is no longer uniform in $u$.
\end{theorem}

Clearly, Theorem \ref{corollary mt differential bis} contains Theorem \ref{theorem mt differential intro} as its special case $s = n/2$. 

\begin{theorem}
\label{corollary mt differential ter}
Let $s \in (0, n)$ and let $Q$ be as in \eqref{Q decay}. Then there is $C > 0$ such that 
\begin{equation}
\label{ineq corollary ter}
\int_\Sph \exp \left( \frac{n + \beta}{|\sph|} K_{n,s}^{-\frac{n}{n-s}}  |u|^\frac{n}{n-s} \right) Q(\eta)  \leq C. 
\end{equation}
uniformly for $u \in C^\infty(\Sph)$ radial with $\|P_{s} u\|_{L^{n/s}(\Sph)} \leq 1$. Moreover, the constant $ \frac{n+\beta}{|\sph|} K_n^{-\frac{n}{n-s}}$ is sharp in the sense that if it is replaced by $\gamma >  \frac{n+\beta}{|\sph|}K_{n,s}^{-\frac{n}{n-s}}$, then the constant $C$ is no longer uniform in $u$.
\end{theorem}

Indeed, Theorems \ref{corollary mt differential bis} and \ref{corollary mt differential ter} follow from Theorem \ref{theorem mt convolution} by expressing $u$ using the Green's function of the operators $P_{2s}^{1/2}$ and $P_s$ respectively. It only remains to check that the resulting convolution integral satisfies \eqref{T definition} and \eqref{expansion K}, with $f = P_{2s}^{1/2} u$ and $f = P_s u$, respectively.

\subsection{Estimates on the Green's function and proof of Theorem \ref{theorem mt differential intro}}
\label{subsection greens}

In this subsection, we prove the following expansion of the Green's functions of $P_{2s}^{1/2}$ and $P_s$, which allows us to deduce Theorem \ref{theorem mt differential intro}, \ref{corollary mt differential bis} and \ref{corollary mt differential ter} from Theorem \ref{theorem mt convolution}. 
We recall that the operators $P_s$ and $P_{2s}^{1/2}$ have been defined in \eqref{paneitz} and \eqref{Pn^12} respectively. 

\begin{lemma}
\label{lemma Green's function}
Let $0 < s \leq \frac{n}{2}$. There exists a function $G_s: \Sph \times \Sph \to \R$ with the following properties. 
\begin{enumerate}
\item For every $u \in C^\infty(\Sph)$ (with $\int_\Sph u = 0$ if $s=\frac{n}{2}$), one has 
\[ u(\eta) = \int_\Sph G_s(\eta, \xi) P_{2s}^{1/2} u(\xi) \diff \sigma(\xi), \]
\item For every $f \in C^\infty(\Sph)$, the function $u(\eta):= \int_\Sph G_s(\eta, \xi) f(\xi)$ satisfies $u \in C^\infty(\Sph)$ (and $\int_\Sph u = 0$ if $s=\frac{n}{2}$) and $P_{2s}^{1/2} u = f$.
\item $G_s(\eta, \xi)$ only depends on $d(\eta, \xi)$. Near the diagonal, one has the expansion 
\begin{equation}
\label{greens estimate}
G_s(\eta, \xi) = K_{n,s} d(\eta, \xi)^{-n+s} \left(1 + \mathcal O\left(d(\eta, \xi)\right)^\alpha \right), 
\end{equation} 
for some $\alpha > 0$. 
\end{enumerate}
Moreover, the same statements hold (without the condition  that $\int u = 0$) if $P_{2s}^{1/2}$ is replaced by $P_s$, and $s \in (0,n)$.
\end{lemma}

For $P_s$ for $s \in (0,n)$, Lemma \ref{lemma Green's function} follows directly from the fact that the Green's function of $P_s$ is explicitly given by 
\begin{equation}
\label{Gs for Ps} 
G_s(\eta, \xi) = K_{n,s} |\eta - \xi|^{-n+s}, 
\end{equation} see e.g. \cite{Be}. 

For $P_{2s}^{1/2}$ and $s \in (0, n/2]$, the argument is a generalization of the case $s=\frac{n}{2}$, see the proof of \cite[Proposition 2.2]{Nd}.

\begin{proof}
Let us assume that $s<\frac{n}{2}$, the proof is readily the same if $s=\frac{n}{2}$ by adding the mean value. Let $w \in C^\infty(S^n)$, we have 
$$w(\eta)=\int_{S^n} G(\eta,\xi) (P_{2s}w)(\xi) \, d\sigma(\xi),$$
where $G$ is the Green function of $P_{2s}$. By self-adjointness of $P_{2s}$ we have
$$w(\eta)=\int_{S^n} (P_{2s}^\frac{1}{2}G)(\eta,\xi) (P_{2s}^\frac{1}{2} w)(\xi) \, d\sigma(\xi).$$
Thus $G_s= P_{2s}^\frac{1}{2}G$ is the function we are looking for. Let us check the estimate along the diagonal. By classical result on pseudo-differential operators, see \cite{AG}, the principal symbols satisfy $\sigma(P_{2s}^\frac{1}{2})=\sigma(P_{2s})^\frac{1}{2}$, hence $P_{2s}^\frac{1}{2}= (-\Delta_\Sph)^{s/2} + \text{lower order operator}$. By using \eqref{Gs for Ps} (respectively $G(\eta,\xi)= K_{n,s} \ln\left(\frac{1}{d(\eta,\xi)}\right)  + O(1)$ for $s = n/2$), we easily deduce the desired estimate.
\end{proof}

%
%
%

\begin{proof}[Proof of Theorem \ref{corollary mt differential bis} and Theorem \ref{corollary mt differential ter}]
We only prove Theorem \ref{corollary mt differential bis}, since the proof of Theorem \ref{corollary mt differential ter} is identical after replacing $P_{2s}^{1/2}$ by $P_s$.
Under the assumptions of Theorem \ref{corollary mt differential bis}, by Lemma \ref{lemma Green's function} we may write 
\begin{equation}
\label{u = Gf} 
u(\eta) = \int_\Sph G_s(\eta, \xi) f(\xi) \diff \xi =: Tf(\eta), 
\end{equation}
where $f := P_{2s}^{1/2} u$ and $G_s(\eta, \xi)$ is the Green's function associated with $P_{2s}^{1/2}$. 

Still by Lemma \ref{lemma Green's function}, the kernel $k(\eta, \xi) = G_s(\eta, \xi)$ satisfies \eqref{expansion K}. Moreover, $\|f\|_{L^{n/s}(\Sph)} = \|P_{2s}^{1/2} u\|_{L^{n/s}(\Sph)} \leq 1$ by assumption. 
Hence 
\[ \int_\Sph \exp \left( \frac{n + \beta}{|\sph|} |u|^\frac{n}{n-s} \right) Q(\eta) \diff \sigma(\eta) = \int_\Sph \exp \left( \frac{n + \beta}{|\sph|} |Tf|^\frac{n}{n-s} \right) Q(\eta) \diff \sigma(\eta) \leq C \]
by Theorem \ref{theorem mt convolution}. 

The sharpness of the constant is again a direct consequence of the sharpness statement in Theorem \ref{theorem mt convolution}. Indeed, let $\gamma > \frac{n+ \beta}{|\sph|}$.Then by Theorem \ref{theorem mt convolution} there exist $f_k$ radial with $\|f_k\|_{L^{n/s}(\Sph)} \leq 1$ such that 
\[ \int_\Sph \exp \left( \gamma |u_k|^\frac{n}{n-s} \right) Q(\eta) \diff \eta \to \infty \]
as $k \to \infty$, where we wrote $u_k :=Tf_k$ with $T f_k$ defined as in \eqref{u = Gf}. Since $u_k$ are radial and $\|P_{2s}^{1/2} u_k\|_{L^{n/s}(\Sph)} = \|f_k\|_{L^{n/s}(\Sph)} \leq 1$, this completes the proof.  
\end{proof}

\section{Proof of Theorem \ref{theorem mt convolution}}

\subsection{An improved Adams' lemma}

The core of the proof of Theorem \ref{theorem mt convolution} is the following one-dimensional calculus lemma stated as Lemma \ref{lemma adams} below. To state it, for $p > 1$ we denote by $p' = \frac{p}{p-1}$ the Hölder conjugate exponent of $p$. 

\begin{lemma}
\label{lemma adams}
Let $1 < p < \infty$, and let $a: \R \times [0, \infty) \to [0,\infty)$ be a function such that
\[ a(w,t) \leq  \begin{cases} 1 + g(w,t) & \text{ for } w \in [0, t], \\
h(w,t) & \text{ for } w \in \R \setminus [0, t], \end{cases}
\]
for some non-negative functions $g(w,t)$ and $h(w,t)$ with the property that
\begin{equation}
\label{g, h assumption adams}
\int_0^t g(w,t) + g(w,t)^{p'} \diff w + \int_{\R \setminus [0,t]} h(w,t)^{p'} \diff w \leq b < \infty, 
\end{equation}
uniformly in $t \in [0, \infty)$. 

For $\phi: \R \to (0,\infty)$ satisfying $\int_{\R} \phi(w)^p \diff w \leq 1$, let 
\[ F(t) := \left( \int_\R a(w,t) \phi(w) \diff w \right)^{p'} - t.  \]
Then for every $\tilde{\alpha} > 0$ there is $C > 0$ not depending on $\phi$ such that 
\[ \int_0^\infty e^{-\tilde{\alpha} F(t)} \diff t \leq C. \]
\end{lemma}

In its basic form, i.e. with $g(w,t) \equiv h(w,t) \equiv 0$, this lemma goes back to Moser \cite{Moser}. The inclusion of the term $h(w,t)$, which is of fundamental importance, is due to Adams \cite[Lemma 1]{Adams}. Adams' version has since been extended in various directions, among others, in \cite{Fontana, AlFeTr, FoMo}. Our new observation is that the error term $g(w,t)$ below may depend on $t$ as long as $g(\cdot,t)$ satisfies suitable integral bounds on the interval $[0,t]$, uniformly in $t > 0$. A typical example of an error term satisfying \eqref{g, h assumption adams} below, and in fact the one we shall use in the simplest case of the proof of Theorem \ref{theorem mt convolution}, would be 
\[ g(w,t) = C e^{-w} + C e^{w-t}, \]
Indeed, \cite{Fontana, FoMo} are able deal with a term of type $e^{-w}$, which arises essentially from the error term in \eqref{expansion K}. The second type $e^{w-t}$ only arises in our case, as a consequence of the lack of a rearrangement argument, as explained in the introduction. Notice that this term cannot be estimated (for $w \in [0,t]$, uniformly in $t > 0$) by any integrable function of $w$ alone. However, as we check in the following, the proof of Adams' lemma allows to include such a term as well. We believe that this observation can be useful in a wider context to prove Moser--Trudinger-type inequalities in settings where rearrangement is not available, notably in the presence of weights.

Our proof follows Adams' paper \cite[proof of Lemma 1]{Adams}. The adaptations we make are close to \cite[proof of Lemma 3.2]{Fontana}. Because of these similarities, we only give a sketch of the proof of Lemma \ref{lemma adams}. The reader may consult \cite{Adams, Fontana} for more details. 

\begin{proof}
For $\lambda \in \R$ we set $E_\lambda := \{ t \geq 0 \, : \, F(t) \leq \lambda \} $ and denote by $|E_\lambda|$ its Lebesgue measure. Since 
\[ \int_0^\infty e^{-\tilde{\alpha} F(t)} \diff t = \tilde{\alpha} \int_\R |E_\lambda| e^{-\tilde{\alpha}\lambda} \diff \lambda, \]
the lemma clearly follows if we can prove the following two assertions. 
\begin{enumerate}
\item[(a)] There is $c > 0$ such that $F(t) \geq -c$ for all $t \geq 0$. Moreover, if $t \in E_\lambda$, then there are constants $d, A, B$ such that 
\begin{equation}
\label{adams sigma estimate}
 (t + d)^{1/p} \left( \int_{\R \setminus [0,t]} \phi(w)^p \diff w \right) ^{1/p} \leq A |\lambda|^{1/p} + B. 
\end{equation}
\item[(b)] There are constants $C, D$ such that for every $\lambda \geq -c$, 
\[ |E_\lambda| \leq C|\lambda| + D. \]
\end{enumerate}
\textit{Proof of (a).  } Let $t \in E_\lambda$. Then by the  definition of $E_\lambda$, the bound on $a(w,t)$ and Hölder's inequality, 
\begin{align*}
& \qquad t - \lambda  \leq \left( \int_\R a(w,t) \phi(w) \diff w \right)^{p'} \\
& \leq \left((1 - L(t)^p)^{1/p} \left( \int_0^t (1 + g(w,t))^{p'} \diff w \right)^{1/p'}  + L(t) \left( \int_{\R \setminus [0,t]} h(w,t)^{p'} \diff w \right)^{1/p'} \right)^{p'} ,
\end{align*}
where we abbreviated 
\[ L(t) := \left( \int_{\R \setminus [0,t]} \phi(w)^p \diff w \right)^{1/p}. \]
The assumptions on $g$ imply that 
\[  \int_0^t (1 + g(w,t))^{p'} \diff s \leq t + d, \]
where $d$ is a constant only depending on $b$. 
We thus obtain 
\[ t - \lambda \leq \left( (1 - L(t)^p)^{1/p} (t + d)^{1/p'} + b L(t)\right)^{p'}. \]
This is the same estimate as \cite[eq. (15)]{Adams}. Now both assertions of (a) follow by arguing as in \cite{Adams}. 

\textit{Proof of (b).  } Suppose that $t_1, t_2 \in E_\lambda$ are such that $\lambda < t_1 < t_2$. We shall show that $t_2 - t_1 \leq C |\lambda|$ for all $\lambda$ large enough, which clearly implies (b). Indeed, similarly to the above 
\begin{align*}
(t_2 - \lambda)^{1/p'} &\leq  \left( \int_0^{t_1} (1 + g(w, t_2))^{p'} \diff w \right)^{1/p'} \left( \int_0^{t_1} \phi(w)^p  \diff w \right)^{1/p} \\
&\quad  + \left( \int_{t_1}^{t_2} (1 + g(w, t_2))^{p'} \diff w \right)^{1/p'} L(t_2)   + \left( \int_{t_1}^{t_2} h(w,t)^{p'} \diff w \right)^{1/2} L(t_2). 
\end{align*}
By the assumptions on $g(w,t)$ we can estimate
\begin{align*}
 \int_0^{t_1} (1 + g(w, t_2))^{p'} \diff w  =  t_1 + d \int_0^{t_2} g(w,t) \diff w + d \int_0^{t_2} g(w,t)^{p'} \diff w \leq t_1 + d
\end{align*}  
and, analogously,
\begin{align*}
\left( \int_{t_1}^{t_2} (1 + g(w, t_2))^{p'} \diff w \right)^{1/p'} \leq (t_2 - t_1 + d)^{1/p'} \leq (t_2 - t_1)^{1/p'} + d^{1/p'}. 
\end{align*}
Combining everything and abbreviating $\delta := (t_2 - t_1)^{1/p'}$, we obtain 
\begin{align*}
t_2 - \lambda  \leq \left( (t_1 + d)^{1/p'} + (\delta + b) L(t_1) \right)^{p'}
\end{align*}
Again, this is the same estimate that occurs in \cite{Adams}. Using the inequality \eqref{adams sigma estimate} from (a) permits to conclude as in \cite{Adams}. 
\end{proof}

\subsection{Transforming to $\R^n$}
\label{subsection transf to Rn}

As a first step in the proof of Theorem \ref{theorem mt convolution}, we recast the setting on $\R^n$ using stereographic projection. 

We first observe $d_{\Sph}(\eta, \xi) = |\eta - \xi|(1 + \mathcal O(|\eta- \xi|^\alpha))$ for suitable $\alpha > 0$ (Here $|\eta - \xi|$ denotes the distance in $\R^{n+1}$.) Together with \eqref{expansion K}, up to taking a smaller $\alpha > 0$, it is easy to deduce from this that $Tf$ defined by \eqref{T definition} satisfies
\begin{equation}
\label{Tf estimate}
Tf(\eta) \leq \int_\Sph |\eta- \xi|^{-n+s} \left(1 + C |\eta- \xi|^\alpha\right) f(\xi) \diff \sigma(\xi). 
\end{equation} 
Now define 
\[ \tilde{f}(x) := f(\mathcal S(x)) J_\mathcal S(x)^{\frac{s}{n}}, \] 
so that $\|\tilde{f}\|_{L^{\frac{n}{s}}(\R^n)} = \|f\|_{L^{\frac{n}{s}}(\Sph)}$. Since the distance transforms as
\[ |\mathcal S(x) - \mathcal S(y)| = J_\mathcal S(x)^{1/2n} |x-y| J_\mathcal S(y)^{1/2n}, \]
by change of variables we obtain from \eqref{Tf estimate} that
\begin{align*}
 \tilde{T} \tilde{f}(x) &:= 
 Tf(\mathcal S(x)) \leq \int_{\R^n} |\mathcal S(x) - \mathcal S(y)|^{-n+s} \left(1 + C|\mathcal S(x) - \mathcal S(y)|^\alpha \right) \tilde{f}(y) J_\mathcal S(y)^\frac{n-s}{n} \diff y \\
&=   \int_{\R^n} |x-y|^{-n+s} \left(\frac{1+|x|^2}{1+|y|^2}\right)^\frac{n-s}{2} \left(1 + C|x-y|^\alpha \left((1+|x|^2)(1+|y|^2)\right)^{-\frac{\alpha}{2}} \right) \tilde{f}(y) \diff y. 
\end{align*} 
By the same change of variables in the integral appearing in Theorem \ref{theorem mt convolution}, we see that it suffices to prove 
\begin{equation}
\label{mt R^n}
\int_{\R^n} \exp \left( \frac{n + \beta}{|\sph|} |\tilde{T} \tilde{f}|^\frac{n}{n-s}  \right) K(x) \diff x \leq C 
\end{equation} 
uniformly in $\tilde{f} \in L^{\frac{n}{s}}_\text{rad}(\R^n)$ with $\|\tilde{f}\|_{\frac{n}{s}} \leq 1$. Here $K(x) := Q(\mathcal S(x)) J_\mathcal S(x)$ satisfies 
\begin{equation}
\label{K growth bound}
K(x) \leq C |x|^\beta \exp(- |x|^\sigma)
\end{equation} 
for some $C, \sigma > 0$, $\beta \in \R$. 

\subsection{Reducing to a local inequality}
\label{subsection local ineq}

In this and the following sections, we will work exclusively in the setting of $\R^n$. For ease of notation, we will drop the tilde and write henceforth, in view of the estimate in the previous section,
\begin{align}
\begin{split}
\label{Tf on Rn}
 Tf(x) &= \int_{\R^n} |x-y|^{-n+s} \left(\frac{1+|x|^2}{1+|y|^2}\right)^\frac{n-s}{2} f(y) \diff y \\
 &\qquad + C \int_{\R^n} |x-y|^{-n+s + \alpha } \frac{(1+|x|^2)^{\frac{n - s - \alpha}{2}}}{(1+|y|^2)^{\frac{n - s +\alpha}{2}}} f(y) \diff y. 
\end{split}
\end{align}
We may and will always assume without explicit mention that $f \geq 0$ and that $\alpha >0$ is so small that $n - s - \alpha > 0$. 

In this step we will use the exponential decay at infinity of the weight $K$ to control the integral in \eqref{mt R^n} in the region $\R^n \setminus B_1$. Indeed, this is a simple consequence of the following pointwise bound. 

\begin{lemma}
\label{lemma Tf leq log r}
Let $Tf$ be defined by \eqref{Tf on Rn} for some $f \in L^{\frac{n}{s}}_\text{rad}(\R^n)$. Then there is $C > 0$ not depending on $f$ such that 
\[ Tf(x) \leq C (1 + \ln (|x|)) \qquad \text{ for all } \, |x| \geq 1.   \]
\end{lemma}

Using this growth bound and \eqref{K growth bound}, we find that  
\[ \int_{\R^n \setminus B_1} \exp \left( \frac{n+\beta}{|\sph|} |Tf|^\frac{n}{n-s}  \right) K(x) \diff x \leq C \int_{\R^n \setminus B_1} \exp \left( C (1 - |x|^{\sigma/2}  \right) \diff x \leq C \]
uniformly in $f$. Hence in what follows, we will only need to prove the local inequality
\[ \int_{B_1} \exp \left( \frac{n+\beta}{|\sph|} |Tf|^\frac{n}{n-s}  \right) K(x) \diff x \leq C. \]

\begin{proof}
[Proof of Lemma \ref{lemma Tf leq log r}]
Let us prove that the first summand of $Tf(x)$ in \eqref{Tf on Rn} satisfied the claimed bound. In radial variables $r = |x|$ and $\rho = |y|$, it reads 
\begin{align*}
 &\quad  \int_0^\infty \left( \int_\sph |re_1 - \rho \omega |^{-n+s} \diff \sigma(\omega)\right) \left(\frac{1+r^2}{1+\rho^2}\right)^{\frac{n-s}{2}} f(\rho) \rho^{n-1} \diff \rho  \\
&= |\sph| r^{-n+s} \int_0^r g_{n-s} \left(\frac{\rho}{r}\right) \left(\frac{1+r^2}{1+\rho^2}\right)^{\frac{n-s}{2}} f(\rho) \rho^{n-1} \diff \rho \\
&\qquad  + |\sph| \int_r^\infty g_{n-s} \left( \frac{r}{\rho} \right) \left(\frac{1+r^2}{1+\rho^2}\right)^{\frac{n-s}{2}} f(\rho) \rho^{s-1} \diff \rho,
\end{align*} 
where $g_{n-s}$ is as in Lemma \ref{lemma mean value} below. First suppose that $s > 1$, then by Lemma \ref{lemma mean value}, $g_{n-s}$ is bounded on $[0,1]$. Thus, using that $\int_0^\infty f(\rho)^{\frac{n}{s}} \rho^{n-1} \diff \rho \lesssim 1$, by Hölder we get that the first summand of $Tf(x)$ is bounded by a constant times
\[  \left( \int_0^r \frac{\rho^{n-1}}{(1+\rho^2)^{n/2}} \diff \rho \right)^{\frac{n-s}{n}} + (1+ r^2)^{\frac{n-s}{2}} \left( \int_r^\infty \frac{\rho^{-1}}{(1 + \rho^2)^{n/2}} \diff \rho \right)^{\frac{n-s}{n}}. \]
Since 
\[  \int_0^r \frac{\rho^{n-1}}{(1+\rho^2)^{n/2}} \diff \rho  \leq C (1 + \ln (r)) \]
and 
\[ \int_r^\infty \frac{\rho^{-1}}{(1 + \rho^2)^{n/2}} \diff \rho \leq C r^{-n} \]
for every $r \geq 1$, we obtain the conclusion. 

If $s < 1$, then by Lemma \ref{lemma mean value}, $g_{n-s}(R) \leq C (1-R)^{s-1}$. Inserting this bound and arguing as before, by Hölder the first summand of $Tf(x)$ is bounded by a constant times
\begin{align*}
&\quad \left( \int_0^r \left(1 - \frac{\rho}{r}\right)^{(s-1)\frac{n}{n-s}} \frac{\rho^{n-1}}{(1+\rho^2)^{n/2}} \diff \rho \right)^{\frac{n-s}{n}} \\
& + (1+ r^2)^{\frac{n-s}{2}} \left( \int_r^\infty \left(1 - \frac{r}{\rho}\right)^{(s-1)\frac{n}{n-s}} \frac{\rho^{-1}}{(1 + \rho^2)^{n/2}} \diff \rho \right)^\frac{n-s}{n}.  
\end{align*} 
Now
\begin{align*}
& \quad \int_0^r \left(1 - \frac{\rho}{r}\right)^{(s-1)\frac{n}{n-s}} \frac{\rho^{n-1}}{(1+\rho^2)^{n/2}} \diff \rho \leq C + \int_1^r  \left(1 - \frac{\rho}{r}\right)^{(s-1)\frac{n}{n-s}} \rho^{-1} \diff \rho \\
&= C +  \int_{1/r}^1 (1-\tau)^{(s-1)\frac{n}{n-s}} \tau^{-1} \diff \tau \leq C (1 + \ln (r)),
\end{align*} 
where we changed variables $\tau = \frac{\rho}{r}$ and noticed that $(s-1)\frac{n}{n-s} > -1$ because $n \geq 2$. Similarly, 
\begin{align*}
&\quad \int_r^\infty \left(1 - \frac{r}{\rho}\right)^{(s-1)\frac{n}{n-s}} \frac{\rho^{-1}}{(1 + \rho^2)^{n/2}} \diff \rho \leq C \int_r^\infty \left(1 - \frac{r}{\rho}\right)^{(s-1)\frac{n}{n-s}} \rho^{-n-1} \diff \rho \\
&= C r^{-n} \int_0^1 (1 - \tau)^{(s-1)\frac{n}{n-s}} \tau^{n-1} \diff \tau \leq C r^{-n}
\end{align*}
by changing variables $\tau = \frac{r}{\rho}$. Altogether, this shows that the first summand of $Tf(x)$ is bounded by $C(1 + \ln (|x|))$. 

The second summand of $Tf(x)$ in \eqref{Tf on Rn} can be treated by the same argument, after simply estimating $(1 + |y|^2)^{-\alpha} \leq 1$ and replacing $s$ by $s + \alpha$.  
\end{proof}

\begin{lemma}
\label{lemma mean value}
For $n \geq 2$, let $0 < \alpha < n$ and define 
\begin{equation}
\label{g definition}
g_\alpha(R) := \frac{1}{|\sph|} \int_\sph |e_1 - R \omega|^{-\alpha} \diff \sigma(\omega). 
\end{equation} 
Then for $R \in [0,1)$, the function $g_\alpha(R)$ satisfies
\[ 
g_\alpha(R) \lesssim 
\begin{cases} 
1 & \text{ if } \alpha \in (0, n-1), \\
1 + \ln \left(\frac{1}{1-R}\right) & \text{ if } \alpha = n-1, \\
(1 - R)^{-\alpha + n -1} & \text{ if } \alpha \in (n-1, n).
\end{cases}
\]
 Moreover, on $[0, 1)$, $g_\alpha(R)$ is 
\begin{itemize}
\item increasing if $\alpha \in (n-2, n)$, 
\item constant if $\alpha = n-2$ and
\item decreasing if $\alpha \in (0, n-2)$.
\end{itemize}
In particular, if $\alpha \in (0, n-2]$, then $\sup_{R \in [0,1]} g_\alpha(R) = g_\alpha(0) = 1$. 
\end{lemma}

\begin{proof}
Let us first prove the claimed bounds on $g_\alpha$. Clearly, $g_\alpha(R)$ is bounded away from $R=1$ for every $\alpha \in (0,n)$. Considering the following picture,

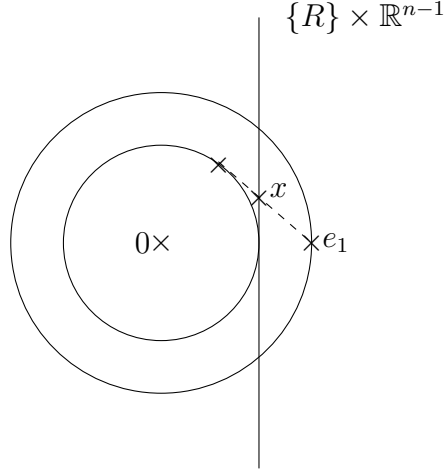
\begin{figure}[H]

\begin{tikzpicture}[scale=2]
\draw (0,0) circle (1);
\draw (0,0) circle (0.65);
\draw (0.65,-1.5)--(0.65,1.5);
\draw (0,0) node {$\times$};
\draw (0,0) node[left] {$0$};
\draw (0.75, 1.5) node[right] {$\{R\}\times \R^{n-1}$};
\draw[dashed] (0.38,0.54) -- (1,0);
\draw (0.65,0.35) node[right] {$x$};
\draw (1,0) node[right] {$e_1$};
\draw (1,0) node {$\times$};
\draw (0.65,0.3) node {$\times$};
\draw (0.38,0.52) node {$\times$};
\end{tikzpicture}
\end{figure}

Then identifying $\{R \} \times \R^{n-1}$ with $\R^{n-1} \times \{0\}$ and writing $d = 1-R$, we clearly have
\begin{align*}
g(R) &\leq \int_{\{ x \in \R^{n-1} \, : |x| \leq 1 \}} |de_n - x|^{-\alpha} \diff x + C \\
&= d^{n-1-\alpha} \int_{\{ x \in \R^{n-1} \, : |x| \leq 1/d \}} |e_n - x|^{-\alpha} \diff x + C \\
&\leq d^{n-1-\alpha}\int_{\{ x \in \R^{n-1} \, : 1 \leq |x| \leq 1/d \}} |x|^{-\alpha} \diff x + C
\end{align*}
As $d \to 0$, up to a constant, the last integral is bounded by 1 if $\alpha < n-1$, by $\ln (1/d)$ if $\alpha = n-1$ and by $d^{\alpha - n +1}$ if $\alpha > n-1$. This yields the claimed bound in each case.

To prove the claimed monotonicity behavior, we observe that the function $g_\alpha(R)$ is the mean value of the function $f(y) := |e_1 - y|^{-\alpha}$ over a sphere around the origin of radius $R$. Now for $r = |x| \neq 0$, 
\[ \Delta (|x|^{-\alpha}) = \left( \frac{\diff^2}{\diff r^2} + \frac{n-1}{r} \frac{\diff}{\diff r}\right) r^{-\alpha} = \alpha ( \alpha + 2 - n) r^{-\alpha - 2}. \]
Hence, in the open ball of radius 1, $f$ is 
\begin{itemize}
\item subharmonic if $\alpha \in (n-2, n)$, 
\item harmonic if $\alpha = n-2$ and
\item superharmonic if $\alpha \in (0, n-2)$.
\end{itemize}
Thus the monotonicity of $g_\alpha$ follows from the mean-value theorem for harmonic functions, respectively the mean-value inequalities for sub- and superharmonic functions.
\end{proof}

\begin{remark}
\label{remark oneils ineq}
The calculations in this section help to further elucidate the discussion following the statement of Lemma \ref{lemma adams} about the new type of error term arising in our setting. We claim that the estimate obtained from of the O'Neil's-type inequality used by Adams in \cite{Adams} does not extend to functions which are radial, but not necessarily decreasing. Indeed, said estimate \cite[bottom of p.390]{Adams} would read, in our notation and for $f \in L^{\frac{n}{s}}_\text{rad}(B_1)$, say, 
\begin{equation}
\label{oneil adams}
 Tf(r) \leq |\sph| \left( \frac{n}{s} r^{-n+s} \int_0^r f(\rho) \rho^{n-1} \diff \rho  + \int_r^1 f(\rho) \rho^{s-1} \diff \rho \right). 
\end{equation}
On the other hand, even when we drop the terms $\frac{1+r^2}{1+\rho^2}$ due to the conformal factors $J_\mathcal S$ (i.e. when we work with the simpler convolution kernel $|x-y|^{-n+s}$ in a setting which does not come from $\Sph$), the calculations in the proof of Lemma \ref{lemma Tf leq log r} yield, for every $f \in L^\frac{n}{s}_\text{rad}(B_1)$, the identity
\begin{equation*}
\label{oneil our version}
Tf(r) = |\sph| \left( r^{-n+s} \int_0^r g_{n-s}\left(\frac{\rho}{r}\right) f(\rho) \rho^{n-1} \diff \rho  + \int_r^1  g_{n-s}\left(\frac{r}{\rho}\right) f(\rho) \rho^{s-1} \diff \rho  \right) . 
\end{equation*}
In the case when $s < 2$, Lemma \ref{lemma mean value} states that $g_{n-s}(R) > g_{n-s}(0)$ for every $R \in (0,1)$. So if $f \equiv 0$ on $B_{r_0}$, for some $r_0 \in (0,1)$, it is plain that \eqref{oneil adams} is violated for every $r \leq r_0$. 
\end{remark}

\subsection{Proof of Theorem \ref{theorem mt convolution}}
\label{subsection proof of thm mt convolution}

In this section, we give the main argument in the proof of Theorem \ref{theorem mt convolution}. Recall that by what we have shown in Sections \ref{subsection transf to Rn} and \ref{subsection local ineq}, it only remains to prove 
\begin{equation}
\label{goal R^n local}
 \int_{B_1} \exp \left( \frac{n + \beta}{|\sph|}  |Tf|^\frac{n}{n-s} \right) |x|^\beta \diff x \leq C, 
\end{equation}
uniformly for all $f \in L^{\frac{n}{s}}_\text{rad}(\R^n)$ with $\|f\|_{L^{\frac{n}{s}}(\R^n)} \leq 1$. Here, $Tf$ is given by \eqref{Tf on Rn}. 

Our strategy is to reduce the proof of \eqref{goal R^n local} to the one-dimensional bound from Lemma \ref{lemma adams}. To achieve this, we pass to radial logarithmic coordinates. Writing $|x| = e^{-t}$, the inequality in \eqref{goal R^n local} transforms to 
\begin{equation}
\label{goal log var}
\int_0^\infty \exp \left( (n+\beta) \left(  \left| \frac{1}{|\sph|^{\frac{n-s}{n}}} (Tf)(e^{-t})\right| ^\frac{n}{n-s}  -  t \right) \right) \diff t \leq C. 
\end{equation}   
Now, define 
\[ \phi(w) := |\sph|^{\frac{s}{n}} f(e^{-w}) e^{-s w}, \]
so that 
\[ \int_\R |\phi(w)|^{\frac{n}{s}} \diff w = \int_{\R^n} |f(x)|^\frac{n}{s} \diff x \leq 1. \]

Similarly to the calculations in Lemma \ref{lemma Tf leq log r}, writing $|x|=e^{-t}$ and $|y|= e^{-w}$, from \eqref{Tf on Rn} we obtain 
\begin{equation}
\label{Tf log var}
\frac{1}{|\sph|^{\frac{n-s}{n}}}(Tf)(e^{-t}) = u(t) + v(t)
\end{equation}
with 
\begin{align*}
u(t) &:= \int_{-\infty}^t g_{n-s}(e^{w-t}) \left( \frac{1+e^{-2t}}{1 + e^{-2w}} \right)^{\frac{n-s}{2}} \phi(w) \diff w \\
&\quad + \int_t^\infty g_{n-s}(e^{t-w}) e^{(n-s)(t-w)} \left( \frac{1+e^{-2t}}{1 + e^{-2w}} \right)^{\frac{n-s}{2}} \phi(w) \diff w 
\end{align*} 
and 
\begin{align*} v(t) &:= C \int_{-\infty}^t g_{n-s-\alpha}(e^{w-t})  \frac{(1+e^{-2t})^\frac{n-s-\alpha}{2}}{(1 + e^{-2w})^\frac{n-s+\alpha}{2}}  \phi(w) \diff w \\
& \quad + C  \int_t^\infty g_{n-s-\alpha}(e^{t-w}) e^{(n-s-\alpha)(t-w)} e^{-\alpha w}  \frac{(1+e^{-2t})^\frac{n-s-\alpha}{2}}{(1 + e^{-2w})^\frac{n-s+\alpha}{2}}  \phi(w) \diff w. 
\end{align*}
We claim that for every $t \geq 0$
\begin{equation}
\label{conv form proof} 
\frac{1}{|\sph|^\frac{n-s}{n}} (Tf)(e^{-t}) = \int_\R a(w,t) \phi(w) \diff w,
\end{equation}
where $a(w,t)$ satisfies the assumptions of Lemma \ref{lemma adams}. Now an application of that lemma (with $\tilde{\alpha} = n + \beta$ and $p = \frac{n}{s}$) concludes the proof of \eqref{goal log var}, and hence of Theorem \ref{theorem mt convolution}.

It thus remains to prove \eqref{conv form proof}. Let us give the bounds for $u(t)$ in detail. For clarity we first treat the simplest case where $s > 1$, so that $g_{n-s} \leq C$ on all of $[0,1]$ by Lemma \ref{lemma mean value}. For $w \in (-\infty, 0]$, we then have 
\[g_{n-s}(e^{w-t}) \left( \frac{1+e^{-2t}}{1 + e^{-2w}} \right)^{\frac{n-s}{2}} \leq C e^{(n-s)w}. \]
For $w \in (0, t]$, we simply estimate $\frac{1+e^{-2t}}{1 + e^{-2w}} \leq 1$. Moreover, for $R = e^{w-t}$ we estimate $g_{n,s}(R) = g_{n,s}(0) + \int_0^R g_{n,s}'(r) \diff r \leq 1 + C R$ if $R$ is near zero and $g(R) \leq C \leq CR$ if $R$ is away from zero. Altogether, this yields  
\[ g_{n-s}(e^{w-t}) \left( \frac{1+e^{-2t}}{1 + e^{-2w}} \right)^{\frac{n-s}{2}} \leq g_{n-s}(0) + C e^{w-t} = 1 + C e^{w-t}. \]
Finally, for $w \in (t, \infty)$, we estimate $g_{n-s}(e^{t-w}) \leq C$ and 
\[  e^{(n-s)(t-w)} \left( \frac{1+e^{-2t}}{1 + e^{-2w}} \right)^{\frac{n-s}{2}} = \left( \frac{e^{2t} +1 }{e^{2w} + 1} \right)^{\frac{n-s}{2}} \leq \left(e^{2(t-w)} + e^{-2w} \right)^{\frac{n-s}{2}} \leq Ce^{(n-s)(t-w)}. \]
All these bounds clearly satisfy the conditions \eqref{g, h assumption adams} of Lemma \ref{lemma adams}. 

If $s \leq 1$, we need to be slightly more careful due to the unboundedness of $g_{n-s}$ near 1. We give the argument for $s < 1$, the adaptation to $s = 1$ is straightforward.  For $w \in (-\infty, -\ln (2)]$, we have $e^{w-t} \leq \frac 12$, so that $g_{n-s}(e^{w-t})$ is uniformly bounded and we can repeat the above estimate. For $w \in (-\ln (2), 0]$ however, we have the modified bound
\[ g_{n-s}(e^{w-t}) \left( \frac{1+e^{-2t}}{1 + e^{-2w}} \right)^{\frac{n-s}{2}} \leq  (1 - e^{w-t})^{s-1} \]
by Lemma \ref{lemma mean value}. But this error term still satisfies \eqref{g, h assumption adams} because 
\[ \int_{-\ln (2)}^0 (1 - e^{w-t})^{(s-1)\frac{n}{n-s}} \leq C, \]
as a consequence of the inequality $(s-1)\frac{n}{n-s} > -1$. Similar arguments justify the validity \eqref{g, h assumption adams} when $w \in (0, t]$ and when $w \in (t, \infty)$.

The term $v(t)$ can be treated by the same argument with $s$ replaced by $s - \alpha$,  after estimating $(1 + e^{-2w})^{-\alpha} \leq 1$ and $e^{-\alpha w}(1 + e^{-2w})^{-\alpha} \leq 1$ for $w \in \R$. 

This proves \eqref{conv form proof}. Hence the proof of Theorem \ref{theorem mt convolution}, up to the sharpness assertion proved in Section \ref{subsection sharpness} below, is complete. \qed

\begin{remark}
In the above argument, we only used the bounds on $g_{n-s}$ given by Lemma \ref{lemma mean value}, but not its monotonicity behavior. In fact, using the latter, our proof simplifies in the case $s \geq 2$. Indeed, in these cases $g(R)$ is bounded and nonincreasing on $R \in [0,1]$ by Lemma \ref{lemma mean value}. This means that we can replace the estimate $g(e^{s-t}) \leq 1 + C e^{s-t}$ from above by the simpler $g(e^{s-t}) \leq g(0) = 1$. The point is that thanks to the absence of a $t$-dependent term like $e^{s-t}$ in the error function $g(w,t)$ from Lemma \ref{lemma adams}, we can use directly Fontana's version of Adams' lemma instead of our improved one; see also the discussion before Lemma \ref{lemma adams}. It is hence only for small orders of derivative $s < 2$ that the new extension of Adams' lemma given in Lemma \ref{lemma adams} is truly decisive. Compare also Remark \ref{remark oneils ineq}. 
\end{remark}

\subsection{Sharpness of the constant}
\label{subsection sharpness}

We now complete the proof of Theorem \ref{theorem mt convolution} by proving the assertion on sharpness of the constant. By \eqref{expansion K} and after projecting to $\R^n$ as in Section \ref{subsection transf to Rn}, let us thus suppose that $\gamma > 0$ is such that
\begin{equation}
\label{mt frac bound greens}
\sup_{0\leq f \in L^{\frac{n}{s}}_\text{rad}(\R^n), \, \|f\|_{\frac{n}{s}} \leq 1} \int_{\R^n} \exp \left( \gamma Tf(\eta)^\frac{n}{n-s} \right) |x|^\beta \exp(-|x|^\sigma) \diff x \leq C < \infty,
\end{equation}
where $Tf$ satisfies
\begin{align}
\label{Tf estimate sharpness}
 Tf(x) &\geq \int_{\R^n} |x-y|^{-n+s} \left(\frac{1+|x|^2}{1+|y|^2}\right)^\frac{n-s}{2} f(y) \diff y  \\
 & \qquad - C \int_{\R^n} |x-y|^{-n+s + \alpha } \frac{(1+|x|^2)^{\frac{n - s - \alpha}{2}}}{(1+|y|^2)^{\frac{n - s +\alpha}{2}}} f(y) \diff y \nonumber
\end{align} 
for some $\alpha > 0$. 

We shall show that necessarily 
\begin{equation}
\label{gamma bound goal}
\gamma \leq \frac{n+\beta}{|\sph|}
\end{equation}
by following the argument of \cite[proof of Theorem 2]{Adams}, where we replace the Riesz potential of $f$ by $Tf$. Let $0 < r < R$ and let $f \in L^{\frac{n}{s}}(B_R)$ be any function such that $Tf \geq 1$ on $B_r$. Testing \eqref{mt frac bound greens} with $f/\|f\|_{\frac{n}{s}}$ gives
\[ C \geq  \exp \left(\gamma \|f\|_{\frac{n}{s}}^{-\frac{n}{n-s}}\right) \int_{B_r} |x|^\beta =  \exp\left(\gamma \|f\|_{\frac{n}{s}}^{-\frac{n}{n-s}}\right) |\mathbb S^{n-1}| \frac{r^{n +\beta }}{n+\beta}, \]
that is,
\begin{equation}
\label{gamma ineq proof} 
\gamma \leq \left( \ln (C) + (\beta+ n) \ln \left(\frac{R}{r}\right) \right) \|f\|_{\frac{n}{s}}^{\frac{n}{n-s}}
\end{equation}
We now make a suitable choice of $f$. For given $r \in (0, R)$, define
\[ f_{r, R}(x) := \begin{cases} |\mathbb S^{n-1}|^{-1} \left(\ln \left(\frac{R}{r}\right)\right)^{-1} |x|^{-s} &\text{ for  } x \in B_R \setminus B_r, \\
0 &\text{ otherwise }. \end{cases} \]
Let $\eps > 0$ be given. In view of the expression of $Tf$ in \eqref{Tf estimate sharpness}, by choosing $R = R(\eps) >0$ small enough we may assume that 
\begin{align*}
Tf_{r,R}(x) &\geq (1- \eps) \int_{B_R} |x-y|^{-n+s} f_{r,R}(y) \diff y  - C \int_{B_R}  |x-y|^{-n+s + \alpha} f_{r,R}(y) \diff y \\
& \geq (1- 2\eps) |\mathbb S^{n-1}|^{-1} \left(\ln \left(\frac{R}{r}\right)\right)^{-1} \int_{B_R \setminus B_r} |x-y|^{-n+s} |y|^{-s} \diff y
\end{align*} 
At fixed $\eps > 0$ and $R = R(\eps) > 0$, it follows from the argument in \cite[p.392]{Adams} (after rescaling by $R$) that there is $r_0(\eps, R) > 0$ such that for every $r < r_0(\eps, R)$, one has 
\[  \int_{B_R \setminus B_r} |x-y|^{-n+s} |y|^{-s} \diff y \geq (1 - \eps) |\sph| \ln\left( \frac{R}{r} \right). \]
Hence,
\[Tf_{r,R}(x) \geq (1- 3\eps) \qquad \text{ for all } x\in  B_r. \]
We conclude that, for $R = R(\eps)$ and every $r < r_0(\eps, R)$, the function $f:= (1-3\eps)^{-1} f_{r, R}$ is an admissible function satisfying $Tf \geq 1$ on $B_r$, and hence \eqref{gamma ineq proof}. Evaluating
\[ \int_{B_R} f_r^{\frac{n}{s}} =  |\mathbb S^{n-1}|^{-\frac{n-s}{s}} (\ln (R/r))^{-\frac{n-s}{s}} , \]
and letting $r \to 0$, inequality \eqref{gamma ineq proof} becomes
\[ \gamma \leq (1- 3 \eps)^{-\frac{n}{n-s}}  \frac{n+\beta}{|\sph|}. \]
Since $\eps >0$ was arbitrary, this completes the proof of \eqref{gamma bound goal}. \qed

\subsection{Proof of Theorem \ref{thm mt local}}

The proof of Theorem \ref{thm mt local} follows the same strategy as that of Theorem \ref{theorem mt convolution}, Theorem \ref{corollary mt differential bis} and Theorem \ref{corollary mt differential ter}. Since moreover it is simpler due to the presence of fewer error terms, we shall be brief. By \cite[Propositions 7 and 8]{Martinazzi2015}, we have the estimate
\[ |u(x)| \leq K_{n,s} \int_{B_R} |x-y|^{-n+s} f(y) \diff y =: K_{n,s} (I_{n-s} \ast f)(x), \] 
where $f := (-\Delta)^{s/2} u|_{B_R}$. (Notice that unless $s$ is an even integer, $(-\Delta)^{s/2} u$ is in general not supported in $B_R$.) Now we are in the situation of Section \ref{subsection proof of thm mt convolution}, only that $Tf$ is replaced by the simpler $I_{n-s} \ast f$. Setting 
\[ \phi(s) = |\sph|^{\frac{s}{n}} f(Re^{-w}) (Re^{-w})^{s}, \]
and changing variables $|x| = R e^{-t}$, $|y| = R e^{-w}$, as above we find
\[ \frac{1}{|\sph|^{\frac{n-s}{n}}}(I_{n-s} \ast f)(Re^{-t}) =  \int_0^t g_{n-s}(e^{w-t}) \phi(w) \diff w + \int_t^\infty g_{n-s}(e^{t-w}) e^{(n-s)(t-w)} \phi(w) \diff w, \]
which satisfies the assumptions of Lemma \ref{lemma adams} by the same estimates as in the proof of Theorem \ref{theorem mt convolution}. Hence
\begin{align*}
&\quad  \int_{B_R} \exp \left( \frac{n+\beta}{|\sph|} K_{n,s}^{-\frac{n}{n-s}} |u|^\frac{n}{n-s} \right) |x|^\beta \diff x \leq \int_{B_R} \exp \left( \frac{n + \beta}{|\sph|}  |I_{n-s} \ast f(x)|^\frac{n}{n-s} \right) |x|^\beta \diff x \\
 &= R^{n+ \beta} |\sph| \int_0^\infty \exp \left( (n+\beta) \left( \left|  \frac{1}{|\sph|^{\frac{n-s}{n}}}(I_{n-s} \ast f)(Re^{-t}) \right|^\frac{n}{n-s} - t \right) \right) \diff t \leq C(R). 
\end{align*}
This completes the proof of Theorem \ref{thm mt local}.

\begin{remark}
\label{remark non-validity on R^n}
It might be tempting to think that analogously to Theorem \ref{theorem mt differential intro}, a 'global' version of inequality \eqref{ineq mt local} might hold on all of $\R^n$ in the presence of an exponentially decaying weight at infinity, i.e. that
\begin{equation}
\label{integral unbounded} 
\int_{\R^n} \exp \left( \gamma |u|^\frac{n}{n-s} \right) |x|^\beta \exp(-|x|^\sigma) \diff x \leq C
\end{equation}
uniformly in $u$ radial with $\|(-\Delta)^{s/2} u\|_{L^{n/s}(\R^n)} \leq 1$. However, this is wrong for any choice of $\gamma, \sigma > 0$ and $\beta \in \R$. A counterexample is given by the family $u_R(x) := (-\Delta)^{-s/2} f_R (x) = K_{n,s} \int_{\R^n} |x-y|^{-n+s} f_R(y) \diff y$, for the $L^\frac{n}{s}(\R^n)$-normalized functions $f_R(y) = (|\sph| \ln( R))^{-\frac{s}{n}} 1_{B_R \setminus B_1}(y)|y|^{-s}$. Indeed, since $|x-y| \leq 2|y|$ whenever $|x| \leq 1 \leq |y|$, we have $u_R(x) \gtrsim (\ln (R))^{\frac{n-s}{n}}$ for every $x \in B_1$. Thus the left side of \eqref{integral unbounded} with $u = u_R$ is unbounded as $R \to \infty$.  
\end{remark}

\end{document}